\documentclass[11pt]{amsart}
\usepackage[a4paper,margin=1in]{geometry}
\usepackage{amsmath,amssymb,mathtools,amsthm}
\usepackage{booktabs}
\usepackage{enumitem}
\usepackage{tikz}
\usetikzlibrary{arrows.meta,positioning,fit,calc,shapes.geometric,shapes.multipart}
\usepackage{microtype}
\usepackage{parskip}
\usepackage{xcolor}
\usepackage[numbers,sort&compress]{natbib}
\usepackage{hyperref}
\hypersetup{colorlinks=true,linkcolor=black,urlcolor=black,citecolor=black}

\title[Reentrant Value Fields]{Reentrant Value Fields as Delayed Coupled Reaction-Diffusion Systems on Finite Graphs}
\author{Karsten Bohlen}
\date{}



\newtheorem{theorem}{Theorem}
\newtheorem{lemma}[theorem]{Lemma}

\newtheorem{proposition}[theorem]{Proposition}
\newtheorem{example}[theorem]{Example}
\newtheorem{blueprint}{Design Specification}
\newtheorem{assumption}[theorem]{Assumption}

\newcommand{\state}{Z}
\newcommand{\leftstate}{H_L}
\newcommand{\rightstate}{X_R}
\newcommand{\controller}{\Theta}
\newcommand{\valuative}{Y}
\newcommand{\frontal}{P}
\newcommand{\memory}{M}
\newcommand{\world}{\mathcal W}
\newcommand{\selfmodel}{\mathcal S}
\newcommand{\meta}{\mathcal U}
\newcommand{\GL}{G_L}
\newcommand{\GR}{G_R}
\newcommand{\VL}{V_L}
\newcommand{\LGL}{\Delta_{\GL}(Q_L)}
\newcommand{\LGR}{\Delta_{\GR}(W_R)}
\newcommand{\LGRhat}{\Delta_{\GR}(\overline{W}_R)}
\newcommand{\Kcal}{\mathcal{K}}
\newcommand{\Xspace}{\mathcal{X}}
\newcommand{\Zspace}{\mathcal{Z}}

\newcommand{\taumax}{\tau_{\max}}
\newcommand{\tauone}{\tau_{R\to L}}
\newcommand{\tautwo}{\tau_{L\to R}}
\newcommand{\Rr}{\mathbb{R}}
\newcommand{\K}{\mathcal{K}}

\begin{document}
\maketitle

\begin{abstract}
We describe a dynamical system in which a symbolic field is coupled to a geometric field via a bipartite Hilbert-Schmidt kernel. The system is fully described by a retarded functional differential equation (RFDE) on the history space, subject to Lipschitz and small gain conditions. We show that the RFDE is well-posed under constant input and that it admits a compact global attractor. The principal subsystem $(H_L, X_R, P)$, which is comprised of the two primary fields as well as an executive field, is shown to be globally stable independent of delay, provided that the interfield coupling satisfies $C_{\mathcal{K}}^2<\mu_L\mu_R$. In addition, we describe design specifications that fulfill the hypotheses of the main Theorem. 
\end{abstract}

\section{Introduction}

We describe a system that is inspired by the cybernetic decision theory of G. E. Pugh \cite{pugh1977}, as well as the concept of reentry \cite{edelman2013reentry}.

This system consists of modules that correspond to perception, valuation, memory, control, and action. These modules are linked by the continuous exchange of latent representations. They can change their own state and share selected hidden content. A module also receives translated content from other components and updates its future dynamics. This repeated exchange leads to the emergence of stable patterns. These patters can influence what is noticed and avoided, as well as what is remembered and forgotten.

The full state vector takes the form
\[
Z(t) = (H_L, X_R, \Theta, Y, P, Q_L, W_R, M, A, \world, \selfmodel, \meta)(t),
\]

where the symbols stand for symbolic state ($H_L$), geometric state ($X_R$), and controller state ($\Theta$). In addition, we have the valuative state ($Y$), executive state ($P$), attention ($Q_L$), awareness ($W_R$) and memory ($M$). Additional elements are the external action ($A$), world model ($\world$), self-model ($\selfmodel$) and meta-control modules ($\meta$), which are further specified in the appendix. The state vector used in the main body of the paper concerns more restricted principal subsystems and is specified in the main Stability section of the paper. 

Each of the modules maintains a history buffer from which it reads delayed coupling signals. The interconnector contains the propagation delays $\tau_{R \to L}$ and $\tau_{L \to R}$. The resulting structure is that of a retarded functional differential equation (RFDE), which is a differential equation whose derivative of the state depends on the current state and on state values at fixed past times, see also \citep{hale1993introduction,hale1988asymptotic,diekmann1995delay}. 

The symbolic and geometric field have the following additional structure: $H_L \in \Rr^{T \times d_L}$ is indexed by sequence positions $\{1, \cdots, T\}$, and $X_R$ is indexed by the nodes of a graph $G_R$. The symbolic field is modeled on the graph $G_L = (V_L, E_L)$, where $V_L = \{1, \cdots, T\}$ are the token-position nodes and the edge weights $Q_L \colon E_L \to \Rr_{> 0}$ are the precision weights. The geometric field similarly is modeled on a graph $G_R = (V_R, E_R)$ with awareness weights $W_R$. In the RFDE formulation, both fields diffuse via weighted Laplacians on finite graphs. They receive delayed coupling signals from each other, and contribute to the principal small-gain margin. The main equation consists of a combination of two graph Laplacians $\Delta_{G_L}(Q_L)$ and $\Delta_{G_R}(W_R)$ plus the bipartite coupling kernel $\K$. 

The paper is organized as follows. The second section gives an overview of the main elements of the framework. In the third section the statements and proofs of the main theorems are given, conditioned by the stated assumptions. Finally, the article contains an extensive appendix which describes possible design choices for the scenario where a potential architecture was to be implemented in practice. We also provide evidence that minimal implementations are possible, e.g. in the examples of the main third section and via appendix \ref{interoperability}, which shows that the adjoint calculus furnishes a natural choice for an implementable neural network that allows the interoperability of the subsystems that are suppose to interact with each other in this framework.  

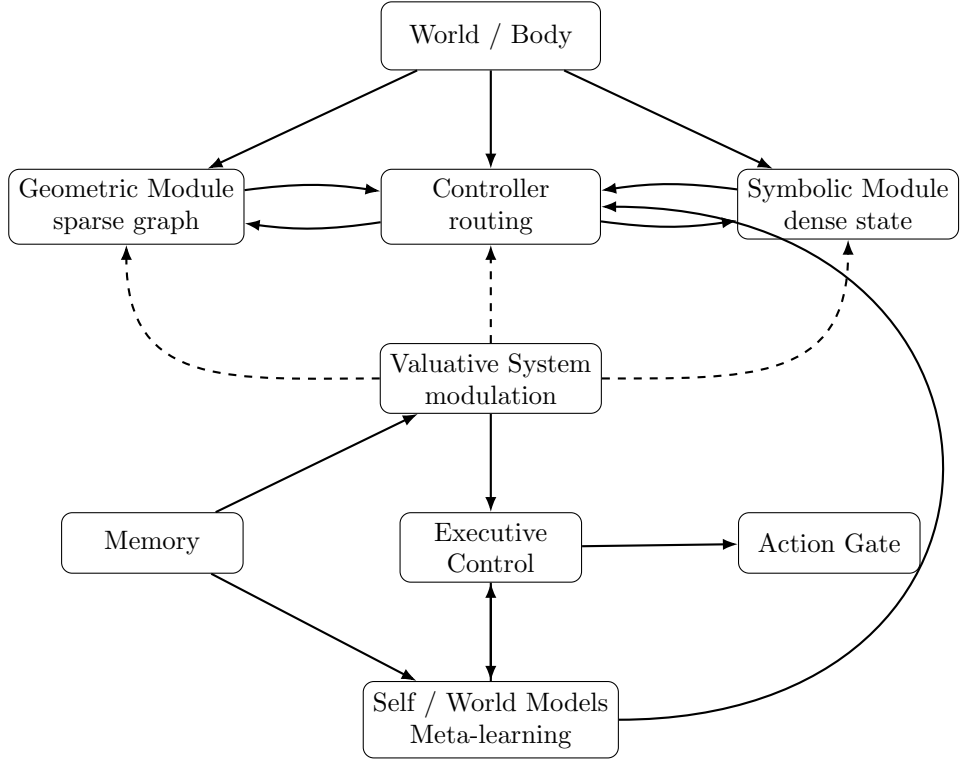
\begin{figure}[ht]
\centering
\begin{tikzpicture}[
  node distance=13mm and 18mm,
  box/.style={draw, rounded corners, align=center, minimum width=29mm, minimum height=9mm,
    font=\small},
  smallbox/.style={draw, rounded corners, align=center, minimum width=24mm,
    minimum height=8mm, font=\small},
  arrow/.style={-{Latex[length=2mm]}, thick},
  mod/.style={-{Latex[length=2mm]}, thick, dashed}
]
\node[box] (env) {World / Body};
\node[box, below left=of env] (right) {Geometric Module\\sparse graph};
\node[box, below=of env] (thal) {Controller\\routing};
\node[box, below right=of env] (left) {Symbolic Module\\dense state};
\node[box, below=of thal] (value) {Valuative System\\modulation};
\node[smallbox, below left=of value] (memory) {Memory};
\node[smallbox, below=of value] (exec) {Executive\\Control};
\node[smallbox, below right=of value] (action) {Action Gate};
\node[smallbox, below=of exec] (meta) {Self / World Models\\Meta-learning};

\draw[arrow] (env) -- (right);
\draw[arrow] (env) -- (thal);
\draw[arrow] (env) -- (left);
\draw[arrow] (right) to[bend left=8] (thal);
\draw[arrow] (left) to[bend right=8] (thal);
\draw[arrow] (thal) to[bend left=8] (right);
\draw[arrow] (thal) to[bend right=8] (left);
\draw[mod] (value.west) to[out=180, in=270, looseness=1.3] (right.south);
\draw[mod] (value.east) to[out=0, in=270, looseness=1.3] (left.south);
\draw[mod] (value) -- (thal);
\draw[arrow] (memory) -- (value);
\draw[arrow] (value) -- (exec);
\draw[arrow] (exec) -- (action);
\draw[arrow] (meta) -- (exec);
\draw[arrow] (meta.east) to[out=0, in=0, looseness=2.2] (thal.east);
\draw[arrow] (memory) -- (meta);
\draw[arrow] (exec) -- (meta);
\end{tikzpicture}
\caption{A minimal reentrant architecture. Solid arrows denote information flow; dashed
arrows denote modulation. The controller selects and routes latent contents. The valuative
system changes gain, learning, and priority.}
\end{figure}

\section{Reaction-Diffusion Model}

The latent representations of the model are treated as fields over representational spaces which are the finite base graphs. The coarse graining that we describe in the appendix preserves the main control variables of the system.

\subsection*{Base graphs}

We define a base space for each of the two primary fields. The idea is to view both fields in a unified manner, by having both live and evolve on finite graphs with possible additional structure that is specific to an architecture.

\subsubsection*{Sequence graph $G_L$}

Let $G_L = (V_L, E_L)$ be a finite connected graph with node set $V_L = \{1, \cdots, T\}$ (token positions) and edge set $E_L$ containing sequential edges $(s, s+1)$ and $(s+1, s)$ for adjacent positions, together with skip edges that encode positional structure. Each edge carries a positional encoding label $p(s, s') \in \Rr^{d_{\mathrm{pos}}}$ which is an architecture constant. 

\subsubsection*{Scene graph $G_R$}

Let $G_R = (V_R, E_R)$ be a finite connected graph with node set $V_R$ (encoding geometric entities) and edge set $E_R$. We endow each edge $(i, j)$ with a fixed rigid-frame label $g_{ij} = (R_{ij}, t_{ij}) \in SE(d)$ which is viewed as an architecture constant. These edge labels encode the geometric relationships among nodes.

\subsection*{Primary fields}

\subsubsection*{Symbolic field}

$H_L \colon V_L \to \Rr^{d_L}$ is a section of the trivial vertex bundle over $G_L$, equivalently $H_L \in \Rr^{T \times d_L}$. The precision weights $Q_L \colon E_L \to \Rr_{> 0}$ define the weighted graph Laplacian
\[
\bigl(\LGL\,\leftstate\bigr)_s
=\sum_{s'\colon(s,s')\in E_L}Q_L(s,s')(h_s - h_{s'}),
\]
which acts as precision-weighted diffusion on $\leftstate$.  In the formal symmetric
regime, connectivity of $\GL$ and $Q_L\geq\varepsilon_Q>0$ give the spectral gap
$\lambda_2(\LGL)\geq\varepsilon_Q\lambda_2(G_L)>0$.

\subsubsection*{Geometric field.}  $\rightstate\colon V\to\mathbb{R}^{d_R}$ is a section of the
trivial vertex bundle over $\GR$, equivalently $\rightstate\in\mathbb{R}^{|V|\times d_R}$.
The rigid-frame labels $g_{ij}$ provide scalar invariant edge features used by the
reaction term, such as $\|t_{ij}\|^2$, traces of rotation features, and prescribed inner
products. The awareness weights
$W_R\in\prod_j\Delta^{k_j}$ define the weighted graph Laplacian
\[
\bigl(\LGR\,\rightstate\bigr)_j
=\sum_{i\colon(i,j)\in E}W_{R,ij}(e_j - e_i),
\]
The associated symmetric conductance field is $\overline W_{ij}=\tfrac12(W_{R,ij}+W_{R,ji})$. We impose that assumption that if $\overline W_{ij}\geq\varepsilon_W$ on a
connected support, then $\lambda_2(\Delta_{\GR}(\overline W))\geq\varepsilon_W\lambda_2(G_R)>0$.

\subsection*{Coupling kernel}

A coupling kernel $\Kcal\colon\VL\times V_R\to\mathbb{R}^{d_L\times d_R}$ induces two
Hilbert-Schmidt operators:
\begin{align*}
& \Kcal \colon L^2(\GR,\mathbb{R}^{d_R})\to L^2(\GL,\mathbb{R}^{d_L}),
&(\Kcal\rightstate)_\ell&=\sum_{i\in V}\Kcal(\ell,i)\,e_i,\\
&\Kcal^* \colon\; L^2(\GL,\mathbb{R}^{d_L})\to L^2(\GR,\mathbb{R}^{d_R}),
&(\Kcal^*\leftstate)_i &= \sum_{\ell\in\VL} \Kcal(\ell,i)^\top h_\ell,
\end{align*}
with HS norm $\|\Kcal\|_{\mathrm{HS}}^2 = \sum_{\ell,i}\|\Kcal(\ell,i)\|_F^2 \leq C_\Kcal^2$.
In practice, the coupling has state-dependent attention weights $\alpha_{\ell i}(Z)\in
\Delta^{|V|-1}$ and $\beta_{i\ell}(Z)\in\Delta^{T-1}$:
\begin{align*}
C_{R\to L,\ell}(t) &= \sum_{i\in V}\alpha_{\ell i}(Z(t))\,\Kcal(\ell,i)\,e_i(t-\tau_{R\to L}),\\
C_{L\to R,i}(t) &= \sum_{\ell\in\VL} \beta_{i\ell}(Z(t))\, \Kcal(\ell,i)^\top h_\ell(t-\tau_{L\to R}).
\end{align*}
The coefficient bound $\alpha_{\ell i} \in [0,1]$ yields
$\|\Kcal_\alpha(Z)\|_{\mathrm{HS}} \leq C_\Kcal$ uniformly in $Z$. The same estimate holds for
$\Kcal^*_\beta(Z)$. For the well-posedness and compactness of the attractor, we also require the Lipschitz bounds
\[
  \|\Kcal_\alpha(Z)-\Kcal_\alpha(Z')\|_{\mathrm{op}} \leq L_\alpha \|Z-Z'\|,
  \qquad
  \|\Kcal_\beta^*(Z)-\Kcal_\beta^*(Z')\|_{\mathrm{op}} \leq L_\beta \|Z-Z'\|.
\]

\subsection*{Master equation}

We collect all components in the state vector
\[
Z = (\leftstate, \rightstate, Q_L, W_R, \mathcal{R}_{\controller}, \valuative, \frontal, \memory,
\{\rho_i\},\{z_i\},\{\theta_i\}) \in \Zspace.
\]
Let $\mathcal E$ be the finite-dimensional Euclidean product space containing all
components of $Z$, and let $\Zspace \subset \mathcal E$ denote the compact admissible
state set. The ambient history space and the admissible phase space are
\[
  \Xspace_{\mathcal E}:=C([-\taumax,0],\mathcal E),
  \qquad
  \Xspace:=C([-\taumax,0],\Zspace)\subset \Xspace_{\mathcal E},
\]
both equipped with the sup norm on histories.  The ambient space
$\Xspace_{\mathcal E}$ is a Banach space and $\Xspace$ is a closed complete
metric subspace. We set
$\taumax=\max(\tau_{R\to L},\tau_{L\to R})$. We start with the RFDE
\begin{equation}\label{eq:RFDE}
\dot{Z}(t) = \mathcal{D}(Z(t)) + \mathcal{N}(Z(t),u(t)) + \mathcal{C}(Z(t),Z(t-\tau_1),Z(t-\tau_2)),
\end{equation}
where $\tau_1=\tau_{R\to L}$ and $\tau_2=\tau_{L\to R}$ are the delays. We separate the right-hand side into the following blocks.

\textit{Diffusion block} $\mathcal{D}(Z)$ is given by
\[
\mathcal{D}(Z) = \bigl(-\LGL \leftstate, -\LGRhat \rightstate, -\kappa_Y\valuative, 0, \ldots\bigr).
\]

\textit{Reaction block} $\mathcal{N}(Z,u)$ contains the instantaneous nonlinear updates
$F_L,F_R,G_Y$. It also contains the auxiliary field updates $\Psi_{Q_L}, \Psi_{W_R}, \Psi_{\mathcal{R}},
\mathcal{P}, \Phi_\memory$, as well as the learning terms
$-\lambda_z z_i+\nabla_{\theta_i}\log\pi_i$,
$-\lambda_{\mathrm{reg}}\theta_i+\eta_i\delta z_i$.
Here every component is evaluated at the current time $t$. To pass to continuous time, we use the residuals of the transformer maps, i.e. $F_L(\leftstate,\ldots):=\mathcal{T}_L(\leftstate,\ldots)-\leftstate$ for the symbolic field, and $F_R(\rightstate,\ldots):=\mathcal{G}_R(\rightstate,\ldots)-\rightstate$ for the geometric field. The remaining auxiliary fields evolve under the continuous-time vector fields
\[
  \Psi_{Q_L},\quad \Psi_{W_R},\quad \Psi_{\mathcal R},\quad \Phi_\memory,
  \quad \Psi_{\rho_i},\quad \Psi_{z_i},\quad \Psi_{\theta_i},
\]
each globally Lipschitz on $\Zspace$ and pointing into the tangent cone of the relevant
compact domain. We obtain the corresponding stagewise rules by explicit Euler or operator-splitting
discretization.

\textit{Delayed coupling block} $\mathcal{C}$ contains all dependence on past state,
\[
\mathcal{C}(Z(t), Z(t-\tau_1), Z(t-\tau_2)) = \bigl(\Kcal_\alpha(Z(t))\, \rightstate(t-\tau_1), \ \Kcal^*_\beta(Z(t))\, \leftstate(t-\tau_2), 0, \ldots\bigr).
\]
The dependence is retarded, rather than neutral. The right-hand side depends on $Z(t+\theta)$ for $\theta\leq 0$ as opposed to $\dot{Z}(t+\theta)$ for $\theta < 0$. We specify an initial history $Z_0\in\Xspace$.

For constant input, the evolution defines a continuous semiflow on $\Xspace$. In the standard history representation, the shift derivative acts on $[-\taumax,0)$ and the RFDE boundary
condition is imposed at zero; cf. \citet{hale1993introduction}, Chapter~2, and \citet{diekmann1995delay}, Chapter~I.

\subsection*{Dissipativity constants and stability condition}

Write $\mu_L,\mu_R,\mu_P>0$ for the one-sided dissipativity constants of the principal
instantaneous $H_L$, $X_R$, and $P$ vector fields. Both $\mu_L$ and $\mu_R$
account for diffusion and reaction contributions; neither includes the delayed
interfield coupling. We take fixed interfield operators with operator norms bounded by $C_\Kcal$ and impose the small-gain condition
\begin{equation}\label{eq:SC}
C_\Kcal^2 < \mu_L\,\mu_R.
\end{equation}
This is the positivity condition for the Schur margin of the coupled principal fields. The executive constant $\mu_P$ enters convergence of the $P$ component separately.

\subsubsection*{Radial margin conditions}

The stability condition~\eqref{eq:SC} controls convergence near equilibrium. A separate
condition controls viability at the compact boundary of the principal balls. At $\|\leftstate\|_F = R_L$, the instantaneous reaction term must counter the largest outward contribution from the delayed coupling. It suffices to require
\begin{equation}\label{eq:radialmargin}
  \eta_L\ge R_L\,C_\Kcal\,R_R,\qquad \eta_R\ge R_R\,C_\Kcal\,R_L,
\end{equation}
where $\eta_L,\eta_R>0$ are the radial damping margins of the reaction terms $F_L$ and
$F_R$ at the respective ball boundaries (i.e.\ $\langle \leftstate,-\Delta_{\GL}(Q_L)\leftstate
+F_L\rangle\le-\eta_L$ when $\|\leftstate\|_F=R_L$, and analogously for $F_R$). The
delayed coupling contributes at most $R_L C_\Kcal R_R$ to the outward radial velocity, so
\eqref{eq:radialmargin} guarantees that the combined inward radial velocity is nonpositive on
the boundary. When the principal fields are implemented as projected relaxation flows, the
margin condition is satisfied automatically by the projection. In the case where they are implemented as explicit Laplacian-plus-reaction systems, the reaction term $F_L$ must be sufficiently
dissipative to enforce~\eqref{eq:radialmargin}.

\subsubsection*{Standing assumption (fixed dimensions and finite delay horizon)}

Throughout this paper, the sequence length $T$, the graph node count $|V|$, the edge count $|E|$, the
subsystem count $n_s$, and the delay scalars $\tau_{R\to L}$, $\tau_{L\to R}\geq 0$ are
fixed finite architecture constants. Both graphs $\GL$ and $\GR$ are assumed to be connected. The
maximum delay is set as $\taumax=\max(\tau_{R\to L},\tau_{L\to R})$. We note that architectures with dynamic token length, variable graph size, or online-adapted delays require a separate analysis.

\subsection*{Role of the design specifications}

Standard known transformer or graph-neural modules fail to fit the required hypotheses of the theory. We therefore include design specifications in the appendix which describe the design choices needed in order to have a stabilized architecture. One instance of such a specification involves the residual backbones which are needed for connectivity and mixing constants for the two primary (symbolic / geometric) fields. These have corresponding overlays that give us the task-dependent selectivity. The viability conditions stated in the main body of the paper are fulfilled by the use of projection onto Frobenius balls. These projections preserve the direction of an overlarge matrix-valued state and rescale its amplitude to the admissible radius. This procedure prevents unbounded amplification and yields a continuous self-map of the compact domain. It does however disregard the magnitude of the state. In practice, the radii are chosen so that projection is normally inactive or weakly active. Smooth norm saturation can be used in place of hard projection. The table lists further mechanisms that might turn a neural network architecture into a stabilized candidate compatible with the hypotheses.

\begin{center}
\begin{tabular}{p{0.38\linewidth}p{0.52\linewidth}}
\toprule
\textbf{Design mechanism} & \textbf{Admissibility contribution} \\
\midrule
Bounded parameterizations &
Compactness of the state domain \(\mathcal Z\). \\[0.35em]

Projected or relaxation-type updates &
Tangent-cone viability and positive invariance of \(\mathcal Z\). \\[0.35em]

Residual symbolic backbone \(Q_L^{\mathrm{base}}\) &
Symbolic spectral gap and dissipativity of the left field. \\[0.35em]

Residual geometric backbone \(W_R^{\mathrm{base}}\) &
Awareness spectral gap and dissipativity of the right geometric field. \\[0.35em]

Sparse learned overlays \(\widetilde Q_L\), \(\widetilde W_R\) &
Task-dependent selectivity without placing the connectivity burden on learned attention or awareness alone. \\[0.35em]

Bounded Hilbert--Schmidt kernels &
Controlled reentrant coupling between symbolic and geometric fields. \\[0.35em]

Small-gain constraints &
Stability of delayed symbolic--geometric feedback loops. \\[0.35em]

Bounded gates and modulation variables &
Controlled variation of coupling, valuation, memory, and policy influence. \\[0.35em]

Simplex-preserving attention, awareness, and routing maps &
Admissible probabilistic weights and preservation of normalized allocation variables. \\[0.35em]

Regularized or projected policy drift &
Bounded adaptive parameters and tangent-cone compatibility for policy variables. \\[0.35em]

Lipschitz neural modules on compact domains &
RFDE well-posedness and continuous dependence on histories. \\[0.35em]

Damping, leak, and residual relaxation terms &
Dissipativity, absorbing estimates, and compatibility with attractor arguments. \\
\bottomrule
\end{tabular}
\end{center}

\section{Stability Criteria}\label{sec:formal_closure}

This section assembles the specification conditions from the appendix into hypotheses which then yield the main results.
On the other hand we also describe the discrete-time implementation which is recorded as a numerical discretization of the RFDE. The distinction between both viewpoints (fields vs. discrete time system) should be kept in mind.

Whenever a conductance field is directed, its graph-diffusion contribution is defined
from the symmetric conductance introduced above. Consequently, every spectral-gap and
graph-Laplacian dissipativity estimate is applied to the same symmetric operator that
appears in the RFDE vector field. The stability constants $\mu_L,\mu_R$ are derived from
these dissipative operators, or are imposed directly as one-sided dissipativity
constants.

\textit{Assumptions.}
The main results rely on the following assumptions.

\begin{assumption}[Fixed finite architecture and finite delays]\label{ass:A1}
The sequence length $T$, graph
  node count $|V|$, edge count $|E|$, subsystem count $n_s$, and delay scalars
  $\tau_{R\to L},\tau_{L\to R}\geq0$ are fixed finite constants.
\end{assumption}

\begin{assumption}[Compact closed state domain]\label{ass:A2}
The state domain $\Zspace$ is the compact product
  domain displayed below, with bounded Euclidean components, floored positive weights, and
  simplex-valued routing components.
\end{assumption}

\begin{assumption}[Global Lipschitz RFDE vector field on histories]\label{ass:A3}
The RFDE right-hand side
  $\mathcal F:\Xspace\to\mathcal E$ is globally Lipschitz on the admissible history
  space $\Xspace=C([-\taumax,0],\Zspace)$ and admits a locally Lipschitz extension\footnote{After componentwise
  Whitney-McShane extension in the finite-dimensional target $\mathcal E$, the Picard theorem 
  applied in $\Xspace_{\mathcal E}$ and the tangent-cone condition below yields the
  admissible subset $\Xspace$.} to
  an open neighbourhood of $\Xspace$ in the ambient Banach history space
  $\Xspace_{\mathcal E}=C([-\taumax,0],\mathcal E)$. 
\end{assumption}

\begin{assumption}[Positive invariance]\label{ass:A4}
The continuous-time vector field points into the tangent
  cone of $\Zspace$ at boundary points.
\end{assumption}

\begin{assumption}[Joint non-emptiness]\label{ass:A5}
The class is nonempty by the minimal construction below. For
richer architectures, we impose a modular joint admissibility condition. This means each added
module must preserve compactness, Lipschitz continuity, tangent-cone viability, and the
relevant dissipativity or bounded-coupling estimates.
\end{assumption}

\begin{assumption}[Equilibrium existence for constant input]\label{ass:A6}
For $u \equiv u^*$ the RFDE has at least one equilibrium $Z^* \in \Zspace$.
\end{assumption}

\begin{assumption}[Closed principal stability regime]\label{ass:A7}
For Theorem~\ref{thm:stability}, the
  delayed interfield operators are fixed bounded Hilbert-Schmidt operators $\Kcal$ and
  $\Kcal^*$ with norms at most $C_\Kcal$. The instantaneous principal dynamics are one-sided dissipative with constants $\mu_L, \mu_R$, and $\mu_P$.
\end{assumption}

\textit{Remark.}
\emph{a)} \textit{Equilibrium existence.}
Suppose the full RFDE has relaxation form $\dot Z=\Gamma(Z_t)(T(Z_t)-Z)$ on
compact convex $\Zspace$. Every fixed point $Z^*=T(Z^*)$ of the stationary target map
$T:\Zspace\to\Zspace$ then gives an equilibrium. For continuous $T$, Brouwer's fixed-point
theorem ensures that at least one such point exists.

\emph{b)} \textit{Viability condition.} Assumption~\ref{ass:A4} is realistic for projected or relaxation-type flows, and for softmax/logit parameterizations of simplex variables. Without further assumptions, delayed coupling and unprojected policy or memory drifts can drive the state outward at the boundary. For every admissible history $\varphi \in C([-\tau_{\mathrm{max}}, 0], \mathcal{Z})$, the condition requires $\mathcal{F}(\varphi, u^{\ast}) \in T_{\mathcal{Z}}(\varphi(0))$. The cone is evaluated at the present state $\varphi(0)$, although $\mathcal{F}$ may depend on the delayed history. For all ball-valued variable, the outward radial velocity at the boundary must be nonpositive. Box-valued variables must point inward at the lower and upper faces. Simplex-valued variables must preserve total mass, with nonnegative velocity at every zero coordinate.\\

\begin{lemma}\label{lem:viability}
Assumption~\ref{ass:A4} decomposes into the following tests, all of which must hold simultaneously at boundary points.
\begin{enumerate}[label=(\roman*)]
\item \emph{Ball component} $x\in B_R$:
  \[
    \|x\| = R \ \Longrightarrow \ \langle x,\dot x \rangle \le 0.
  \]
\item \emph{Box component} $q \in [a,b]$:
  \[
    q=a \ \Rightarrow \ \dot q \ge 0, \qquad q=b \ \Rightarrow \ \dot q \le 0.
  \]
\item \emph{Simplex component} $w\in\Delta^m$:
  \[
    \textstyle \sum_i \dot w_i = 0, \qquad w_i = 0\ \Rightarrow\ \dot w_i \ge 0.
  \]
\end{enumerate}
\end{lemma}

\begin{lemma}\label{lem:rfde_viability}
Let $\mathcal E$ be finite-dimensional, let $\Zspace\subset\mathcal E$ be closed and
convex, and let $\Xspace=C([-\taumax,0],\Zspace)$.  Suppose
$\mathcal F:\Xspace\to\mathcal E$ is continuous, admits a locally Lipschitz extension to
the ambient history space $\Xspace_{\mathcal E}=C([-\taumax,0],\mathcal E)$, and satisfies
\[
  \mathcal F(\varphi)\in T_{\Zspace}(\varphi(0))
  \qquad\text{for every }\varphi\in\Xspace .
\]
Then every classical RFDE solution of $\dot Z(t)=\mathcal F(Z_t)$ with initial history in
$\Xspace$ remains in $\Zspace$ for all forward times for which it exists.
\end{lemma}

\begin{proof}
We apply the Seifert positive-invariance criterion in the retarded setting, cf. \citep{seifert1976}. At time $t$, the velocity is determined by the history $Z_t$; the boundary condition, however, is tested at the present state
$Z(t)=Z_t(0)$. Closedness and convexity of $\Zspace$, together with the tangent-cone condition, prevent a first exit through a boundary face. For the product cone in Lemma~\ref{lem:viability}, this yields positive invariance of $\Zspace$.
\end{proof}

\medskip

\begin{proposition}\label{prop:relaxation}
Let $K_a$ be compact convex and let $T_a \colon \Xspace \to K_a$ be continuous and bounded. If every
constrained component of the RFDE is either of the form
\begin{equation}\label{eq:relax}
  \dot x_a = \gamma_a(Z_t) \bigl(T_a(Z_t) - x_a\bigr), \qquad \gamma_a \ge 0,
\end{equation}
or the projected form $\dot x_a = \gamma_a(Z_t)(\Pi_{K_a} \hat T_a(Z_t) - x_a)$, then the full
vector field satisfies Assumption~\ref{ass:A4}.
\end{proposition}

\begin{proof}
Convexity of $K_a$ ensures that $T_{K_a}(x_a)$ contains every direction $y-x_a$ with $y\in K_a$.
Thus the component derivative $\dot x_a = \gamma_a(T_a-x_a)$ belongs to $T_{K_a}(x_a)$ whenever
$T_a \in K_a$. Projection onto $K_a$ preserves this property. The tangent cone of a product
of closed convex sets is the product of their tangent cones, so the full derivative belongs to
$T_{\Zspace}(Z)$. For a ball, the calculation is explicit: $\|x_a\| = R$ and $\|T_a\| \le R$ imply
$\langle x_a, T_a - x_a\rangle \le R^2 - R^2=0$.
\end{proof}

\subsection{Formal closure}

We next chck that the specifications also give a well-defined discretized system with the required closure.

The formal map treats $A^t \in \mathcal{A}$ as an exogenous input from the environment
or by a separate sampling procedure. Given $({\tilde{Z}^t}, u^t, A^t)$, the map
$\tilde{\mathcal{F}}$ is deterministic. The architecture contains action policy parameters
$\theta_{\rm act}$ and a policy $\pi_{\rm act}$ governing the sampling of $A^t$. The
analysis of these parameters depends on the externally supplied action. In that sense the agent is an open system, because perceptual inputs $u^t$ arrive from outside and actions $A^t$ exit to perturb the environment. A fully closed agent-environment loop lies outside the scope of the present
analysis.

Recall the full state vector
\[
Z = \bigl(\leftstate,\ \rightstate,\ Q_L,\ W_R,\ \mathcal{R}_\controller,\ \valuative,\ 
\frontal,\ \memory,\ \{\rho_i\},\ \{z_i\},\ \{\theta_i\}\bigr) \in \Zspace.
\]

The three classes of auxiliary variable have the following description.
\emph{State-augmented variables:} the reliability variables $\rho_i$, eligibility traces
$z_i$, and policy parameters $\theta_i$ are dynamic state variables with their own update
rules and bounds and are included in $Z$ above.
\emph{Stagewise-derived scalars:} the homeostatic deviation $h$, prediction error
$\varepsilon_{\mathrm{pred}}$, novelty signal $n$, and outcome feedback $r$ are functions
of $(u^t,A^t,Z^t)$ computed at each step.
\emph{Delay scalars and history buffer:} the delay scalars
$\tau_{R\to L}$ and $\tau_{L\to R}$ are fixed architecture constants; the history buffer is
the finite discretized tail used by the implementation.
\emph{Controller broadcast:} the broadcast $B_\controller^t$ is computed from
$\mathcal{R}_\controller^t\in Z^t$ and subsystem export vectors at the opening of the
update step.
\emph{Memory gate:} the gate $g_\memory$ is computed as a deterministic Lipschitz function
of $Z^t$ at Stage~8.5.

The variable types are summarized in the following table.

\begin{center}
\small
\begin{tabular}{p{0.22\linewidth}p{0.44\linewidth}p{0.24\linewidth}}
\hline
Type & Examples & Treatment \\
\hline
Dynamic state & $\leftstate, \rightstate, Q_L, W_R, \mathcal{R}_\controller,
  \valuative, \frontal, \memory, \rho_i, z_i, \theta_i$ &
  $\tilde{Z}^t$, updated each step \\
Stagewise-derived & $h, \varepsilon_{\mathrm{pred}}, n, r, B_\controller, g_\memory$ &
  Computed from $Z^t$ \\
External inputs & $u^t$ (sensory), $A^t$ (action) &
  Supplied exogenously \\
Architecture constants & $\tau_{R\to L},\tau_{L\to R},K,\varepsilon_Q,R_L,R_R,C_W$ &
  Fixed during a rollout \\
History buffer & $\{Z^{t-k}\}_{k=1}^{K}$ &
  Discretized tail of history \\
Edge attributes & $g_{ij} \in \mathrm{SE}(d)$ &
  Fixed graph labels \\
\hline
\end{tabular}
\end{center}

The state domain $\mathcal{Z}$ is characterized by the following explicit constraints:
\[
\leftstate \in B_{R_L}, \quad e_i \in B_{R_R}, \quad Q_L \in [\varepsilon_Q,R_Q]^{T\times T},
\quad W_R \in \textstyle \prod_j \Delta^{k_j},\quad
\mathcal{R}_\controller \in \textstyle \prod_{i=1}^{n_s} \Delta^{n_s-1},
\]
\[
\valuative \in B_{R_Y}, \quad \frontal\in B_{R_P}, \quad \memory \in B_{R_M},\quad
\rho_i \in[0,1], \quad z_i \in B_{R_{z,i}}, \quad \theta_i \in B_{R_{\theta,i}}.
\]
The constants $\varepsilon_Q>0$, $R_Q$, $R_L$, $R_R$, $R_Y$, $R_P$, $R_M$, $R_{z,i}$, and
$R_{\theta,i}$ are finite parameters of the architecture. The radius $R_M$ satisfies
$R_M\geq\max(\|\memory^0\|,C_\memory)$.

On the stated domain, each component of $Z$ has the following Lipschitz vector field:
\[
\begin{array}{lll}
\leftstate & -\Delta_{\GL}(Q_L)\leftstate+F_L(Z,u)+\Kcal_\alpha(Z)\rightstate(t-\tauone) &
\text{principal field},\\
\rightstate & -\Delta_{\GR}(W_R)\rightstate+F_R(Z,u)+\Kcal_\beta^*(Z)\leftstate(t-\tautwo) &
\text{principal field},\\
Q_L & \Psi_{Q_L}(Z) & \text{compact positive box},\\
W_R & \Psi_{W_R}(Z) & \text{product of simplices},\\
\mathcal R_\controller & \Psi_{\mathcal R}(Z) & \text{routing simplex product},\\
\valuative & -\kappa_Y\valuative+G_Y(\leftstate,\rightstate,u) & \text{valuative field},\\
\frontal & \mathcal P(Z) & \text{executive field},\\
\memory & \Phi_\memory(Z) & \text{memory field},\\
\rho_i & \Psi_{\rho_i}(Z) & \text{reliability variable},\\
z_i & \Psi_{z_i}(Z) & \text{eligibility trace},\\
\theta_i & \Psi_{\theta_i}(Z) & \text{projected smooth policy drift}.
\end{array}
\]
The auxiliary fields are globally Lipschitz on $\Zspace$ and obey the tangent
cone condition in (A4). We discretize the continuous-time equation by the stagewise algorithm below. 

The subsequent theorems use
Specifications~\ref{bp:symbolic}--\ref{bp:memory}, together with the admissibility class $\mathfrak{C}_P$.

\textit{Remark (Discrete-time implementation).}
We implement the continuous-time RFDE~\eqref{eq:RFDE} by explicit Euler or operator splitting in the stages below. Introduce the augmented state $\tilde{Z}^t:=(Z^t,Z^{t-1},\ldots,Z^{t-K})$. The corresponding shift map is $\tilde{\mathcal{F}}(\tilde{Z}^t,u^t,A^t):=(Z^{t+1},Z^t,\ldots,Z^{t-K+1})$. Assumption (A5) determines $Z^{t+1}$ uniquely in an acyclic order. As a consequence, $\tilde{\mathcal{F}}$ defines a deterministic map on $\tilde{\mathcal{Z}}:=\mathcal{Z}^{K+1}$.

The update evaluates the components of $Z^{t+1}$ in eight stages. At each stage the
inputs are either elements of $Z^t$, elements of $(u^t,A^t)$, or outputs of earlier stages
within the same step.

\textbf{Stage 1 (Neuromodulatory signals).} Compute $\mu=\mathcal{V}_\mu(\valuative^t)\in(0,1)^5$
from $\valuative^t$ using the neuromodulatory readout component of $\mathcal{V}$.
Well-definedness follows from Specification~\ref{bp:valuative}(ii)-(iii). The single input
$\valuative^t$ is an element of $Z^t$.

\textbf{Stage 2 (Precision and awareness fields).} The base logit field $a_L\in\mathbb{R}^{T\times T}$
and the modulation weight $b_L\in\mathbb{R}^{T\times T}$ (with $b_L\geq 0$) are fixed
architecture constants. Use
$q_L^t:=a_L+\mu_{\mathrm{ACh}}^t\,b_L$ for the modulated logit field at step $t$. It depends deterministically on the architecture constants and on $\mu_{\mathrm{ACh}}^t$ obtained from Stage~1.
We compute
$Q_L^{t+1}=\varepsilon_Q+(R_Q-\varepsilon_Q)\,\sigma(q_L^t)$
from Stage~1 according to Specification~\ref{bp:symbolic}(v)-(viii). For awareness, take the logit field $\omega^t:=\omega(Z^t)$, which depends deterministically on $Z^t$. We then compute
$W_R^{t+1}$ from $\omega^t$ and $\mu_{\mathrm{NE}}^t$ (Stage~1), as specified in
Specification~\ref{bp:geometric}(v)-(vii). Lipschitz continuity of the respective maps makes both updates well-defined. 

\textbf{Stage 3 (Interconnector signals).} The delay discretization yields the integer index $n=\lfloor\tau/\Delta t\rfloor\geq 1$. Define
$\operatorname{Delay}_\tau(\rightstate^t,\ldots,\rightstate^{t-n}):=\rightstate^{t-n}$
(or use the weighted combination). We obtain
$C_{R\to L}=g_{R\to L}^t\,\Phi_{R\to L}(\operatorname{Delay}_\tau(\rightstate^t,\ldots,\rightstate^{t-n}))$
from the history buffer and $\mathcal{R}_\controller^t$. In the other direction, we obtain
$C_{L\to R}=g_{L\to R}^t\,\Phi_{L\to R}(\operatorname{Delay}_\tau(\leftstate^t,\ldots,\leftstate^{t-n}))$
from the history buffer and $\mathcal{R}_\controller^t$. The gates are already available as entries of
$\mathcal{R}_\controller^t\in Z^t$. Specification~\ref{bp:interconnector}(i)-(iii) furnishes boundedness.

\textbf{Stage 4 (Symbolic state).} We form $\leftstate^{t+1}=\mathcal{T}_L(\leftstate^t,
B_\controller^t,C_{R\to L},Q_L^{t+1},\frontal^t,\memory^t)$ from the outputs of Stages~2-3 and the state $Z^t$.
Well-definedness follows from Specification~\ref{bp:symbolic}(i).

\textbf{Stage 5 (Geometric state).} Compute $\rightstate^{t+1} = \mathcal{G}_R(\rightstate^t,
B_\controller^t,C_{L\to R},W_R^{t+1},\valuative^t,\memory^t)$ using Stages~2-3 and $Z^t$.
Well-definedness follows from Specification~\ref{bp:geometric}(i).

\textbf{Stage 6 (Valuative state).} Compute $\valuative^{t+1} = \mathcal{V}_Y(\valuative^t,
h^t, \varepsilon_{\mathrm{pred}}^t, n^t, r^t, \leftstate^{t+1}, \rightstate^{t+1}, \memory^t,\frontal^t)$ which uses the outputs of Stages~4-5, where $h^t, \varepsilon_{\mathrm{pred}}^t, n^t, r^t$ are stagewise-derived scalars computed from $(u^t,A^t,Z^t)$. Then well-definedness follows by virtue of Specification~\ref{bp:valuative}(iii)-(iv).

\textbf{Stage 7 (Routing matrix).} We set $\mathcal{R}_\controller^{t+1} = K(\leftstate^{t+1},\rightstate^{t+1},\valuative^{t+1}, \rho^t, Z^t_{\mathrm{rest}})$.
Here $Z^t_{\mathrm{rest}}$ collects the components of $Z^t$ that have not yet been updated.
At this stage, $K$ acts on the partial-update product space
$B_{R_L} \times \prod_i B_{R_R} \times B_{R_Y} \times [0,1]^{n_s} \times \mathcal{Z}_{\mathrm{rest}}$. This agrees with Specification~\ref{bp:controller}(iii), provided we interpret that specification on the routing inputs available at Stage~7. The reliability scores $\rho^t\in Z^t$ are already available and need no additional update here. Specification~\ref{bp:controller}(iii) makes this step well-defined.

\textbf{Stage 8 (Reliability, parameters, executive, memory).}
We carry out Stage~8 in the fixed order below. At each sub-step, all required quantities are already available:
\begin{enumerate}[label=8.\arabic*,leftmargin=*]
\item Compute $\delta$ from $\valuative^{t+1}$ (Stage~6) according to Specification~\ref{bp:policy}(iii);
  $\delta\in[-D,D]$.
\item Compute $\rho_i^{t+1}=(1-\alpha)\rho_i^t+\alpha f(\varepsilon_i^t)$ according to
  Specification~\ref{bp:reliability}(i)-(iii). This is well-defined: $\rho_i^t\in Z^t$ is known, and
  the stagewise-dervied scalar $\varepsilon_i^t$ is available at the start of the step.
\item Compute $z_i^{t+1} = \lambda_i z_i^t+\nabla_{\theta_i} \log\pi_i^\varepsilon$ and
  $\Delta \theta_i^t = \eta_i(\delta z_i^t - \lambda_{\mathrm{reg}} \theta_i^t)$;
  then project $\theta_i^{t+1} \leftarrow \Pi_{B_{R_{\theta,i}}}(\theta_i^t + \Delta\theta_i^t)$.
  This makes use of $z_i^t$ (as opposed to the just-computed $z_i^{t+1}$) and the one-step-delayed REINFORCE convention.
\item Compute $\frontal^{t+1}=\mathcal{P}(\frontal^t,\{\Delta\theta_i^t\},\valuative^{t+1})$
  using Specification~\ref{bp:executive}(i)-(iii). Uses $\Delta\theta_i^t$ from sub-step 8.3 and
  $\valuative^{t+1}$ from Stage~6; does not use $\memory^{t+1}$.
\item Compute $\memory^{t+1}=(1-g_\memory)\memory^t+g_\memory\Phi_\memory(\leftstate^{t+1},
  \valuative^{t+1})$ using Specification~\ref{bp:memory}(i)-(iii). Uses $\leftstate^{t+1}$
  (Stage~4) and $\valuative^{t+1}$ (Stage~6); does not use $\frontal^{t+1}$ or $z_i^{t+1}$.
\end{enumerate}

At each stage and sub-stage the required inputs are in $\tilde{Z}^t$, $(u^t,A^t)$, or outputs of earlier stages. These checks denomonstrate that the within-step dependency graph is acyclic.

\textit{Remark (Covariant closure).}
The RFDE establishes that $Z(t)$ evolves continuously from $Z_0 \in \Xspace$. A stronger
condition would be the covariant closure which would require the dynamics to be consistent under the symmetries of the physical environment. Specifically, if $\phi \in \mathrm{Aut}(G) \times \mathrm{SE}(d)$ is applied to the physical scene, the field trajectories of the transformed and untransformed
systems should be related by $\phi$ throughout. Since the rigid-frame transforms $g_{ij}$ are
fixed structural labels and the node embeddings $e_i$ are scalar and $\mathrm{SE}(d)$-invariant
features, the relevant symmetry property is permutation-equivariance under $\mathrm{Aut}(G)$
combined with $\mathrm{SE}(d)$-invariance in the feature sector. The full covariant closure, would additionally require the interconnector to intertwine the symmetry action. By Specification~\ref{bp:geometric}(ii), the geometric update
$\mathcal{G}_R$ satisfies: $\mathcal{G}_R(\sigma\cdot\{e_i\},\{g_{\sigma(i)\sigma(j)}\}) = \sigma \cdot\mathcal{G}_R(\{e_i\}, \{g_{ij}\})$
for any $\sigma\in\mathrm{Aut}(G)$, and the outputs are invariant to $\mathrm{SE}(d)$
relabelings of $g_{ij}$, cf. \citet{satorras2021egnn}.
Covariant closure of the full system would additionally require that the interconnector
intertwines this action with a corresponding representation $\rho_L$ of
$\mathrm{Aut}(G)\times\mathrm{SE}(d)$ on the symbolic state space, i.e.
$\Phi_{R\to L}(\phi\cdot\{e_i\})=\rho_L(\phi)\cdot\Phi_{R\to L}(\{e_i\})$, and that the
symbolic update operator $\mathcal{T}_L$ commutes with $\rho_L$. If this requirement is not fulfilled, the geometric module is equivariant in isolation but the equivariance is not transmitted through the interconnector, and the agent's symbolic dynamics may represent different tokens for physically equivalent scenes in different frames. Whether a consolidated architecture satisfying all of these conditions simultaneously exists is an open question. The conditions above identify what full covariant closure requires and are left for further development.

The following table summarizes the dependency structure.

\begin{center}
\begin{tabular}{lll}
\hline
Component & Principal inputs & Specification \\
\hline
$\mu^{t+1}$ & $\valuative^t$ & \ref{bp:valuative} \\
$Q_L^{t+1}$ & $\mu_{\mathrm{ACh}}$ (Stage~1) & \ref{bp:symbolic} \\
$W_R^{t+1}$ & $\rightstate^t,\;\mu_{\mathrm{NE}}$ (Stage~1) & \ref{bp:geometric} \\
$C_{R\to L},C_{L\to R}$ & $\rightstate^{t-\tau},\;\leftstate^{t-\tau},\;
  \mathcal{R}_\controller^t$ & \ref{bp:interconnector} \\
$\leftstate^{t+1}$ & $Q_L^{t+1},\;C_{R\to L},\;\text{Stages 1-3}$ & \ref{bp:symbolic} \\
$\rightstate^{t+1}$ & $W_R^{t+1},\;C_{L\to R},\;\text{Stages 1-3}$ & \ref{bp:geometric} \\
$\valuative^{t+1}$ & $\leftstate^{t+1},\;\rightstate^{t+1},\;h^t,\varepsilon_{\mathrm{pred}}^t,n^t,r^t$ &
  \ref{bp:valuative} \\
$\mathcal{R}_\controller^{t+1}$ & $\leftstate^{t+1},\;\rightstate^{t+1},\;\valuative^{t+1},\;
  \rho^t$ & \ref{bp:controller} \\
$\rho_i^{t+1}$ & $\varepsilon_i^t$ (prediction errors of subsystems) & \ref{bp:reliability} \\
$\frontal^{t+1}$ & $\mu_{\mathrm{DA}},\;z_i^t,\;\theta_i^t,\;\text{Stages 4-7}$ & admissibility of $\mathcal{P}$ \\
$z_i^{t+1},\;\theta_i^{t+1}$ & $z_i^t,\;\theta_i^t,\;\delta^t,\;\pi_i^\varepsilon$ & \ref{bp:policy} \\
$\memory^{t+1}$ & $\leftstate^{t+1},\;\valuative^{t+1}$ &
  \ref{bp:memory} \\
\hline
\end{tabular}
\end{center}



\subsection{Well-posedness and stability}

\begin{theorem}\label{thm:wellposed}
Assume ~\ref{ass:A1}--\ref{ass:A5}. Given constant input $u\equiv u^*$ and every initial history $\varphi\in\Xspace=C([-\taumax,0],\Zspace)$, the RFDE~\eqref{eq:RFDE} has a unique forward
solution $Z\in C([-\taumax,\infty),\Zspace)$. The history-segment map
$\varphi\mapsto Z_t$ defines a continuous semiflow $\{T(t)\}_{t\geq0}$ on $\Xspace$ 
\end{theorem}

\begin{proof}
Apply the Picard theorem on the ambient Banach history space
$\Xspace_{\mathcal E}=C([-\taumax,0],\mathcal E)$. By Assumption~\ref{ass:A3}, the RFDE vector field
$\mathcal F \colon \Xspace \to \mathcal E$ is globally Lipschitz on admissible histories and has
a locally Lipschitz extension to an open neighbourhood in $\Xspace_{\mathcal E}$. By virtue of the
delayed coupling, we have
\begin{align*}
&\bigl\|\Kcal_\alpha(\varphi(0))\varphi_R(-\tauone) -\Kcal_\alpha(\psi(0))\psi_R(-\tauone)\bigr\|\\
&\leq C_\Kcal\|\varphi_R(-\tauone)-\psi_R(-\tauone)\| + L_\alpha R_R \|\varphi(0)-\psi(0)\|,
\end{align*}
and the reverse delayed term is identical with $L_\beta$ and $R_L$. We first obtain a unique local solution in the ambient space from the standard RFDE existence theorem. The
RFDE viability lemma, together with Assumption~\ref{ass:A4}, keeps this solution in $\Zspace$ throughout its interval of existence. Since $\Zspace$ is compact, Assumption~\ref{ass:A2} excludes finite escape. The solution therefore extends globally. For constant input, uniqueness and the shift of history segments yield the semiflow property.
\end{proof}

\begin{theorem}\label{thm:attractor}
Assume~\ref{ass:A1}--\ref{ass:A5} and constant input. The semiflow $\{T(t)\}_{t\geq0}$ has a compact
global attractor $\mathcal A\subset\Xspace$. 
\end{theorem}

\begin{proof}
Compactness and positive invariance of $\Zspace$ furnishes a uniform bound on all solution values. Compactness of the history space is obtained by a uniform boundedness argument. Only finitely many history evaluations enter the right-hand side. Each component map is continuous and bounded on the compact state domain $\Zspace$, so the RFDE vector field is uniformly bounded on the admissible histories. It follows that the solution segments $T(t)\varphi$ with $t>\taumax$ are uniformly Lipschitz on $[-\taumax,0]$. Applying Arzela-Ascoli, we obtain precompactness of $T(t)B$ for every bounded set $B\subset\Xspace$ and every $t>\taumax$. We can therefore apply the attractor theorem for eventually compact dissipative semiflows \cite{hale1988asymptotic} to obtain a compact global attractor.
\end{proof}

\begin{theorem}\label{thm:stability}
Assume Assumptions~\ref{ass:A1}--\ref{ass:A7}, constant input, fixed auxiliary variables
\[
Q_L,W_R,Y,M,\mathcal R_\controller,\rho_i,z_i,\theta_i
\]
at their principal-regime equilibrium values, fixed interfield operators $\Kcal$ and
$\Kcal^*$, and the small-gain condition~\eqref{eq:SC}. Then the principal equilibrium
$(H_L^*,X_R^*,P^*)$ is unique and globally asymptotically stable for all finite delays
$\tauone$, $\tautwo\geq0$. The Lyapunov-Krasovskii functional
\begin{align*}
  V(\tilde\varphi) &= \tfrac{1}{2}\|\tilde\varphi_L(0)\|^2
  +\tfrac{1}{2}\|\tilde\varphi_R(0)\|^2
  +\tfrac{1}{2}\|\tilde\varphi_P(0)\|^2\\
  &\quad+\tfrac{C_\Kcal^2}{2\mu_L}\!\int_{-\tauone}^{0}\!\|\tilde\varphi_R(s)\|^2\,ds\\
  &\quad+\tfrac{C_\Kcal^2}{2\mu_R}\!\int_{-\tautwo}^{0}\!\|\tilde\varphi_L(s)\|^2\,ds
\end{align*}
satisfies along solutions
\[
  \dot V\leq
  -\alpha_L\|\tilde\leftstate(t)\|^2
  -\alpha_R\|\tilde\rightstate(t)\|^2
  -\mu_P\|\tilde\frontal(t)\|^2,
\]
where $\alpha_L=\mu_L/2-C_\Kcal^2/(2\mu_R)>0$ and
$\alpha_R=\mu_R/2-C_\Kcal^2/(2\mu_L)>0$.
\end{theorem}

\begin{proof}
In the closed principal regime, we fix the auxiliary variables and the interfield operators
$\Kcal$ and $\Kcal^*$.  After centering, the principal equations contain only the delayed
coupling terms $\Kcal\tilde X_R(t-\tauone)$ and
$\Kcal^*\tilde H_L(t-\tautwo)$. The definition of this reduced regime excludes state-dependent Laplacian variations and auxiliary-field variations. 

To prove uniqueness of the principal equilibrium, let
\[
  h=H_1-H_2,\qquad x=X_1-X_2,\qquad p=P_1-P_2
\]
denote the difference of two principal equilibria. Subtract their equilibrium equations and test against $(h,x,p)$. We obtain
\[
\begin{aligned}
0
&\leq
-\mu_L\|h\|^2-\mu_R\|x\|^2-\mu_P\|p\|^2
 + \langle \Kcal x,h\rangle+\langle \Kcal^*h,x\rangle  \\
&\leq
-\mu_L\|h\|^2-\mu_R\|x\|^2-\mu_P\|p\|^2
 +2C_\Kcal\|h\|\,\|x\|.
\end{aligned}
\]
The quadratic form
\[
  \mu_L a^2+\mu_R b^2-2C_\Kcal ab
\]
is positive definite because $C_\Kcal^2<\mu_L\mu_R$. It follows that $h=x=p=0$, which proves uniqueness.

For stability, differentiate $V$ term by term. Dissipativity bounds the instantenous contributions by
$-\mu_L\|\tilde H_L\|^2$, $-\mu_R\|\tilde X_R\|^2$, and
$-\mu_P\|\tilde P\|^2$. By virtue of Young's inequality
\begin{align*}
\langle \Kcal\tilde X_R(t-\tauone),\tilde H_L(t)\rangle
&\leq \frac{C_\Kcal^2}{2\mu_L}\|\tilde X_R(t-\tauone)\|^2 + \frac{\mu_L}{2}\|\tilde H_L(t)\|^2,
\end{align*}
and
\begin{align*}
\langle \Kcal^*\tilde H_L(t-\tautwo),\tilde X_R(t)\rangle
&\leq \frac{C_\Kcal^2}{2\mu_R}\|\tilde H_L(t-\tautwo)\|^2 + \frac{\mu_R}{2}\|\tilde X_R(t)\|^2 .
\end{align*}
The derivatives of the two integral terms cancel the delayed norms exactly. The current terms that remain yield the stated inequality. Positivity of $\alpha_L,\alpha_R,\mu_P$ then makes the current principal errors square integrable on $[0,\infty)$. They are also uniformly continuous, since the RFDE vector field is bounded on the compact invariant set. By virtue of Barbalat's lemma we therefore obtain convergence of $(H_L(t),X_R(t),P(t))$ to the principal equilibrium. To control stability, use $V(t)\le V(0)$. There is a constant $C_\tau>0$, depending only on
$\tauone,\tautwo,C_\Kcal,\mu_L,\mu_R$, for which
\begin{align*}
&V(0)\le C_\tau\|\tilde\varphi\|_\infty^2, \\
&V(t)\ge
\tfrac12\bigl(
\|\tilde H_L(t)\|^2+
\|\tilde X_R(t)\|^2+
\|\tilde P(t)\|^2
\bigr).
\end{align*}
These bounds show that a sufficiently small initial history keep the current principal state uniformly small. When $0\le t\le\taumax$, the initial history controls the rest of the history segment. Once $t>\taumax$, the segment contains only solution segments already controlled by these bounds.
This proves Lyapunov stability in the history sup norm. Applying convergence of the current
state over each translated delay interval then yields $Z_t\to Z^*$ in the history norm.
\end{proof}

\textit{Remark.}
\emph{i)} Under the hypotheses of Theorem~\ref{thm:stability}, allow
$\Kcal_\alpha(Z)$, $\Kcal_\beta^*(Z)$, $\Delta_{G_L}(Q_L)$, and
$\Delta_{G_R}(W_R)$ to depend Lipschitz-continuously on $Z$, with Lipschitz constants
$L_\alpha$, $L_\beta$, $L_Q$, and $L_W$, respectively. When we center at an equilibrium $Z^*$, the fixed-coupling delayed terms are accompanied by the instantaneous perturbations
\begin{align*}
&  \bigl(\Kcal_\alpha(Z)-\Kcal_\alpha(Z^*)\bigr)X_R^*,\qquad
   \bigl(\Kcal_\beta^*(Z)-\Kcal_\beta^*(Z^*)\bigr)H_L^*,\\
&  \bigl(\Delta_{G_L}(Q_L)-\Delta_{G_L}(Q_L^*)\bigr)H_L^*,\\
&  \bigl(\Delta_{G_R}(W_R)-\Delta_{G_R}(W_R^*)\bigr)X_R^* .
\end{align*}
Thus, in the centered $\tilde H_L$ equation, the instantaneous perturbation is bounded by
\[
  m_L\|\tilde Z\|,
  \qquad
  m_L:=L_\alpha\|X_R^*\|+L_Q\|H_L^*\|,
\]
and in the centered $\tilde X_R$ equation the corresponding bound is
\[
  m_R\|\tilde Z\|,
  \qquad
  m_R:=L_\beta\|H_L^*\|+L_W\|X_R^*\|.
\]
To conclude global stability conclusion for the fully state-dependent adaptive system, we still need an
auxiliary Lyapunov or input-to-state stability (ISS) estimate for the remaining components of $\tilde Z$. The resulting coupled dissipation matrix must also be positive.

\emph{ii)} Under Theorem~\ref{thm:stability}, the cross-gain is zero. The auxiliary variables, including $\valuative$, policy deltas, routing, and memory, are fixed at equilibrium. Thus only intrinsic dissipativity contributes to the centered executive block:
\[
  \frac{d}{dt}\frac12\|p\|^2
  \le -\mu_P\|p\|^2 .
\]
If active executive forcing is retained, write
\[
  \|R_P(h,x,y)\|
  \le
  c_{PH}\|h\|+c_{PX}\|x\|+c_{PY}\|y\|,
\]
where all $c_{\cdot}$ are nonnegative gain bounds. With the valuative relaxation
\[
  \dot{\valuative}
  =
  \kappa_Y\bigl(\Phi_Y(\leftstate,\rightstate,u)-\valuative\bigr),
\]
define
\[
  R_\Phi(h,x)
  :=
  \Phi_Y(\leftstate^*+h,\rightstate^*+x,u^*)
  -
  \Phi_Y(\leftstate^*,\rightstate^*,u^*),
\]
and assume
\[
  \|R_\Phi(h,x)\|
  \le
  L_{\Phi H}\|h\|+L_{\Phi X}\|x\|.
\]
Valuation may therefore respond to symbolic and geometric perturbations only with bounded
gain. The fast valuative coordinate tracks $y\approx R_\Phi(h,x)$, so the gains transmitted through valuation to the executive block are
\[
  c_{PY}L_{\Phi H},
  \qquad
  c_{PY}L_{\Phi X}.
\]
Consequently define
\[
  c_{PH}^{\mathrm{eff}}
  :=
  c_{PH}+c_{PY}L_{\Phi H},
  \qquad
  c_{PX}^{\mathrm{eff}}
  :=
  c_{PX}+c_{PY}L_{\Phi X}.
\]
Absorbing the delayed symbolic-geometric terms leaves the margins
\[
  \omega_L
  :=
  \frac{\mu_L}{2}-\frac{C_\Kcal^2}{2\mu_R},
  \qquad
  \omega_R
  :=
  \frac{\mu_R}{2}-\frac{C_\Kcal^2}{2\mu_L}.
\]
The executive cross-gain condition is then
\[
  M_P^{\mathrm{eff}}
  :=
  \begin{pmatrix}
    \omega_L & 0 & -\frac12 c_{PH}^{\mathrm{eff}}\\[2mm]
    0 & \omega_R & -\frac12 c_{PX}^{\mathrm{eff}}\\[2mm]
    -\frac12 c_{PH}^{\mathrm{eff}} & -\frac12 c_{PX}^{\mathrm{eff}} & \mu_P
  \end{pmatrix}
  >0 .
\]

For finite $\kappa_Y$, we interpret this condition together with the usual $O(\kappa_Y^{-1})$ fast valuative relaxation.

\begin{theorem}\label{thm:slowfast}
Assume constant input, bounded trajectories, $\kappa_Y>0$, and let
\[
  \phi(t):=\Phi_Y(H_L(t),X_R(t),u^*)
\]
along a complete bounded solution.  Suppose
\[
  \phi\in W^{1,\infty}_{\mathrm{loc}}([0,\infty)),
  \quad
  \|\phi(t)\|\le R_Y,
\]
\[
  \|\dot\phi(t)\|\le G_1
  \quad\text{for a.e. }t\ge0 .
\]
Consider the relaxation equation
\[
  \dot Y(t)=\kappa_Y(\phi(t)-Y(t)).
\]
Then
\[
  \|Y(t)-\phi(t)\|
  \le
  e^{-\kappa_Y t}\|Y(0)-\phi(0)\|
  +
  \frac{G_1}{\kappa_Y}.
\]
Thus, after the exponentially decaying transient, the valuative coordinate lies within
$O(\kappa_Y^{-1})$ of the quasi-steady graph
\[
  Y=\Phi_Y(H_L,X_R,u^*).
\]
\end{theorem}

\begin{proof}
Introduce the tracking error
\[
  W(t):=Y(t)-\phi(t).
\]
The regularity $\phi\in W^{1,\infty}_{\mathrm{loc}}$ yields, for a.e. $t$
\[
  \dot W(t)
  =
  -\kappa_Y W(t)-\dot\phi(t).
\]
By variation of constants we obtain
\[
  W(t) = e^{-\kappa_Y t}W(0) - \int_0^t e^{-\kappa_Y(t-s)}\dot\phi(s)\,ds .
\]
Hence
\begin{align*}
  \|W(t)\|
  &\le
  e^{-\kappa_Y t}\|W(0)\|
  +
  G_1\int_0^t e^{-\kappa_Y(t-s)}\,ds \\
  &\le
  e^{-\kappa_Y t}\|Y(0)-\phi(0)\|
  +
  \frac{G_1}{\kappa_Y}.
\end{align*}
\end{proof}

By Specification~\ref{bp:interconnector}(ii), $C_{R\to L,\ell}=\sum_{i\in V}\alpha_{\ell i}(Z)\mathcal{K}(\ell,i)e_i$
with $\mathcal{K}(\ell,i)\in\mathbb{R}^{d_L\times d_R}$ and $e_i\in\mathbb{R}^{d_R}$, so
$C_{R\to L,\ell}\in\mathbb{R}^{d_L}$ and $C_{R\to L}\in\mathbb{R}^{T\times d_L}$, matching
the token dimension of $\leftstate$. By Specification~\ref{bp:interconnector}(iii),
$C_{L\to R,i}=\sum_{\ell=1}^T\beta_{i\ell}(Z)\mathcal{K}(\ell,i)^\top H_{L,\ell}$
with $\mathcal{K}(\ell,i)^\top\in\mathbb{R}^{d_R\times d_L}$ and $H_{L,\ell}\in\mathbb{R}^{d_L}$,
so $C_{L\to R,i}\in\mathbb{R}^{d_R}$ and $C_{L\to R}\in\mathbb{R}^{|V|\times d_R}$, matching
the node embedding space. In Specification~\ref{bp:controller}, routing compatibility is scalar-valued for $Z_i,Z_j\in\mathbb{R}^{d_Z}$.

\textit{Remark.}
\emph{i)} Product extension allows us to add admissible auxiliary modules. Suppose an auxiliary
state \(A\in K_A\) has a globally Lipschitz vector field compatible with its tangent cone. If it enters the core equations through bounded Lipschitz maps whose gains
stay within the small-gain budget, then the enlarged system remains admissible.

\emph{ii)} The $W^{1,\infty}$ hypothesis in Theorem \ref{thm:slowfast} follows in the smooth case provided $G_Y$ has bounded first derivative on the compact state domain and the RFDE vector field is uniformly bounded. With Lipschitz or nonsmooth architectures, we retain it as a regularity assumption along the particular trajectory.\\

\begin{example}
\emph{(a)} Let the symbolic graph be the complete graph \(G_L=K_3\) and the geometric
graph be the path graph \(G_R=P_3\), with Laplacians
\[
L_L=
\begin{pmatrix}
2&-1&-1\\
-1&2&-1\\
-1&-1&2
\end{pmatrix},
\qquad
L_R=
\begin{pmatrix}
1&-1&0\\
-1&2&-1\\
0&-1&1
\end{pmatrix}.
\]
Thus \(G_L\) is dense while \(G_R\) is sparse.  Let
\[
H(t),X(t)\in\mathbb R^3,
\qquad
Y(t),P(t)\in[-1,1],
\]
and fix the nonzero equilibrium
\[
H^*=X^*=e_1=(1,0,0)^\top,
\qquad
Y^*=P^*=0.
\]

Define state-dependent conductances
\[
Q(Y)=1+\delta_Q\tanh(Y),
\qquad
W(P)=1+\delta_W\tanh(P),
\]
\[
0<\delta_Q,\delta_W<1.
\]
Equivalently,
\[
Q(Y)=Q^{\mathrm{base}}+\widetilde Q(Y),
\qquad
Q^{\mathrm{base}}=1-\delta_Q,
\]
\[
\widetilde Q(Y)=\delta_Q(1+\tanh Y)\geq0,
\]
and
\[
W(P)=W^{\mathrm{base}}+\widetilde W(P),
\qquad
W^{\mathrm{base}}=1-\delta_W,
\]
\[
\widetilde W(P)=\delta_W(1+\tanh P)\geq0.
\]
Thus we have the dense symbolic backbone \(K_3\) and the sparse geometric backbone \(P_3\). The state-dependent overlays \(\widetilde Q(Y)\) and \(\widetilde W(P)\) provide modulation.  Since
\[
\lambda_2(L_L)=3,
\qquad
\lambda_2(L_R)=1,
\]
we have
\[
\lambda_2(Q(Y)L_L)\geq 3(1-\delta_Q)>0,
\]
\[
\lambda_2(W(P)L_R)\geq 1-\delta_W>0.
\]

Let
\[
K_0=
\begin{pmatrix}
1&1/2&0\\
0&1&1/2\\
1/2&0&1
\end{pmatrix}.
\]
Then
\[
\|K_0\|_{\mathrm{HS}}^2=\frac{15}{4},
\qquad
\|K_0\|_{\mathrm{HS}}=\frac{\sqrt{15}}2.
\]
Define gated reentrant kernels
\[
K_\alpha(Y)=g_\alpha(Y)K_0,
\qquad
g_\alpha(Y)=k+\sigma_\alpha\tanh(Y),
\]
and
\[
K_\beta^*(P)=g_\beta(P)K_0^\top,
\qquad
g_\beta(P)=k+\sigma_\beta\tanh(P).
\]
This furnishes a one-channel gated mixture kernel with fixed translation channel
\(K_0\).  A conservative uniform coupling bound, using
\[
\|K_0\|_{\mathrm{op}}\leq \|K_0\|_{\mathrm{HS}},
\]
is
\[
C_{\mathcal K}
=
\frac{\sqrt{15}}2\bigl(k+\max\{\sigma_\alpha,\sigma_\beta\}\bigr).
\]

Choose instantaneous reaction residuals
\[
F_L(H)
=
-\alpha_H(H-H^*)+L_LH^*-kK_0X^*,
\]
and
\[
F_R(X)
=
-\alpha_X(X-X^*)+L_RX^*-kK_0^\top H^*.
\]
Then the RFDE
\[
\dot H(t)
=
-Q(Y(t))L_LH(t)
+
F_L(H(t))
+
K_\alpha(Y(t))X(t-\tau_1),
\]
\begin{align*}
\dot X(t) &= -W(P(t))L_RX(t) + F_R(X(t)) \\
&\quad + K_\beta^*(P(t))H(t-\tau_2),
\end{align*}
\[
\dot Y(t)=-\kappa_Y Y(t),
\qquad
\dot P(t)=-\kappa_P P(t)
\]
has \(Z^*=(H^*,X^*,0,0)\) as an equilibrium.

Indeed, at \(Y=P=0\),
\[
Q(0)=W(0)=1,
\quad
K_\alpha(0)=kK_0,
\quad
K_\beta^*(0)=kK_0^\top,
\]
and the definitions of \(F_L,F_R\) give
\[
-L_LH^*+F_L(H^*)+kK_0X^*=0,
\]
\[
-L_RX^*+F_R(X^*)+kK_0^\top H^*=0.
\]
Moreover,
\[
-\kappa_Y Y^*=0,
\qquad
-\kappa_P P^*=0.
\]

Centering at \(Z^*\) produces the additional state-dependent terms
\[
(K_\alpha(Y)-K_\alpha(0))X^*,
\qquad
(Q(Y)L_L-L_L)H^*,
\]
and
\[
(K_\beta^*(P)-K_\beta^*(0))H^*,
\qquad
(W(P)L_R-L_R)X^*.
\]
Their Lipschitz constants may be bounded by
\[
L_\alpha\leq \sigma_\alpha\|K_0\|_{\mathrm{HS}}
=
\frac{\sqrt{15}}2\sigma_\alpha,
\]
\[
L_\beta\leq \sigma_\beta\|K_0\|_{\mathrm{HS}}
=
\frac{\sqrt{15}}2\sigma_\beta,
\]
and
\[
L_Q=\delta_Q\|L_L\|_{\mathrm{op}}=3\delta_Q,
\qquad
L_W=\delta_W\|L_R\|_{\mathrm{op}}=3\delta_W.
\]
Since \(\|H^*\|=\|X^*\|=1\), the state-dependent perturbation load is
\[
M_{\mathrm{sdc}}
=
\max\left\{
\frac{\sqrt{15}}2\sigma_\alpha+3\delta_Q,\;
\frac{\sqrt{15}}2\sigma_\beta+3\delta_W
\right\}.
\]
Taking
\[
\mu_L=\alpha_H,
\qquad
\mu_R=\alpha_X,
\]
the uniform delayed small-gain condition for the principal block is
\[
C_{\mathcal K}^2<\alpha_H\alpha_X.
\]
The quantity $M_{\mathrm{sdc}}$ above bounds the auxiliary-to-principal perturbation generated
by deviations of $Y$ and $P$ from equilibrium. Since
$\dot Y=-\kappa_YY$ and $\dot P=-\kappa_PP$ in this example, these perturbations decay
exponentially. Adding their exponentially decaying forcing terms to the
Lyapunov-Krasovskii estimate and applying Young's inequality yields convergence of
$(H,X)$ under the displayed uniform small-gain condition. For fixed
$\alpha_H,\alpha_X>0$, this condition holds on a nonempty open parameter region whenever
$k,\sigma_\alpha,\sigma_\beta$ are sufficiently small; $\delta_Q,\delta_W$ remain subject
to the positivity conditions on $Q(Y)$ and $W(P)$.

\emph{(b)} The auxiliary equations in the previous example may be replaced by bounded dissipative
valuation-executive dynamics
\[
\dot Y(t)
=
-\kappa_Y Y(t)
+
a_Y\tanh\bigl(r_Y(H(t),X(t),P(t))\bigr),
\]
and
\[
\dot P(t)
=
-\kappa_P P(t)
+
a_P\tanh\bigl(r_P(H(t),X(t),Y(t))\bigr),
\]
where \(r_Y\) and \(r_P\) are bounded Lipschitz functions. To preserve the displayed
equilibrium \(Y^*=P^*=0\), choose the functions centered at \((H^*,X^*)\).  For instance,
in the \(K_3/P_3\) model one may take
\[
r_Y(H,X,P)
=
\frac13\sum_{i=1}^3(h_i-h_i^*)
+
\frac13\sum_{i=1}^3(x_i-x_i^*)
+
P,
\]
and
\begin{align*}
r_P(H,X,Y)
&=
\bigl((h_1-h_3)-(h_1^*-h_3^*)\bigr)\\
&\quad+
\bigl((x_1-x_3)-(x_1^*-x_3^*)\bigr)
+
Y.
\end{align*}
Then
\[
r_Y(H^*,X^*,0)=0,
\quad
r_P(H^*,X^*,0)=0,
\]
so the equilibrium \(Z^*=(H^*,X^*,0,0)\) is retained.  The first is a centered
coarse symbolic-geometric activation signal modulated by executive context, while the
second is a centered endpoint-contrast signal modulated by valuation.

This modification is compatible with the compact viability assumptions.  If
\[
Y,P\in[-1,1],
\qquad
0<a_Y\leq \kappa_Y,
\qquad
0<a_P\leq \kappa_P,
\]
then, since \(|\tanh|\leq1\),
\[
Y=1\Rightarrow
\dot Y\leq -\kappa_Y+a_Y\leq0,
\]
\[
Y=-1\Rightarrow
\dot Y\geq \kappa_Y-a_Y\geq0,
\]
and similarly
\[
P=1\Rightarrow
\dot P\leq -\kappa_P+a_P\leq0,
\]
\[
P=-1\Rightarrow
\dot P\geq \kappa_P-a_P\geq0.
\]
Thus the intervals \([-1,1]\) for \(Y\) and \(P\) remain positively invariant.  The
additional terms are bounded Lipschitz auxiliary reactions. Stability of the resulting active
$(H,X,Y,P)$ system requires a separate block small-gain estimate that includes the
$P\to Y$ and $Y\to P$ gains. The conclusion established here is positive invariance of the
auxiliary intervals.
\end{example}

\section{Extensions and Interpretations}

\textit{Continuum limit.} A continuum limit would require to specify
conditions for graph-convergence as $T\to\infty$ and $|V|\to\infty$, a scaling under which
$\Delta_{\GL}(Q_L)$ and $\Delta_{\GR}(W_R)$ converge to limiting differential or integral
operators, convergence of the reentrant kernels $\Kcal$, and dissipativity estimates uniform
in the finite approximants. Under such additional hypotheses one may seek a limiting PFDE
on a space such as
\[
  L^2([0,1],\mathbb{R}^{d_L})\times L^2(\Omega,\mathbb{R}^{d_R}),
\]
together with upper semicontinuity of attractors. See also \citep{wu1996partial}.

\textit{Stochastic extension.} To extend the RFDE stochastically, we would need a noise model that respects the constrained state domain $\Zspace$. Possible choices are reflecting, projected, or tangent noise. Another route would be to formulate the dynamics on an unconstrained dissipative state
space. In infinite dimensions, we must also choose the noise covariance so that the stochastic convolution takes values in the state space. Further questions concern the existence of invariant measures, concentration in the small-noise limit, and the required corrections to the Lyapunov-Krasovskii functional.

\textit{Specification compliance and constraints.} In the appendix we state the conditions an architecture must satisfy and derive consistency
properties using only those conditions. Reentrant instability is the central difficulty. A module can receive its own output after a chain of translations, delays, and gain modulations; the resulting feedback may amplify perturbations rather than damp them. This is particularly relevant when the controller routing matrix concentrates weight on a single feedback path and the delay $\tau$ is small compared with the receiving module's characteristic relaxation time. One safeguard is to bound outputs unconditionally, as in Specification~\ref{bp:symbolic}(ii) and
Specification~\ref{bp:geometric}(iv). Such bounds do not directly constrain the gain of the full closed loop. We must therefore analyze stability of the coupled field equations for each implementation separately. Discrete-time integration is numerically stiff wherever diffusion coefficients are large relative to the integration step. These cases require implicit solvers or carefully chosen step sizes.  Both smooth and non-smooth implementations occur within each specification class. Non-smooth implementations within a class, including piecewise-affine
awareness projections and sparse routing operators with support boundaries, satisfy all
the specification conditions and formal results. Gradient-based
learning nevertheless requires additional care to avoid instability of their support.

\textit{Neurogeometry and the geometric module.} The Geometric Module section \ref{app:geo} describes equivariant message passing as a high-level abstraction of a specific
cortical substrate. A potentially more anatomically faithful description of the dorsal visual
processing stream comes from neurogeometry \cite{petitot2003}. In this description, the primary visual cortex is a fiber bundle over the retinal surface.
The fiber at each point of the retinotopic map records the local orientation preference.
The contact structure on the bundle is inspired by a physiological observation:
cortical neurons preferring the same orientation at adjacent retinal positions are
preferentially connected by short-range lateral interactions. Cortical columns
form the local fibers. At a fixed retinotopic location, a hypercolumn spans the full cycle of orientation preferences. Long-range horizontal connections
between hypercolumns implement parallel transport of orientation signals across the
retinotopic plane. We interpret the rigid-frame transforms $g_{ij}\in\mathrm{SE}(d)$ as
this fixed anatomical contact structure. The frame relation between two adjacent
hypercolumns is determined by the cortical wiring and does not change from moment to
moment, which is consistent with their treatment as fixed edge labels in the
architecture. The sparse awareness field $W_R$, in turn, describes the changing
selectivity of long-range horizontal projections: which connections are active
depends on the ongoing computation. A cortical
interpretation using retinotopy, orientation
columns, hypercolumns, and long-range horizontal connectivity is plausible. Developing it in full lies outside the scope of this paper \citep{kobayashi1963foundations,bronstein2021geometric}.

\textit{Tradeoff between valuative and intellectual complexity.}
\citet[Ch.~11]{pugh1977} points out a systematic inverse relationship between the resolution
of the value system and the effective range of intellectual deliberation. A high dimensional, finitely differentiated value organization, i.e. one with many drives, many
homeostatic variables, and many neuromodulatory channels, focuses intellectual resources on
a narrow set of options. It resolves behavioral ambiguity quickly but reduces
flexibility. With a sparsely specified value organization, more
options remain open to deliberation; competing actions that receive similar valuations then introduce indeterminacy. In our architecture, the dimensionality and precision of the neuromodulatory vector
$\mu=(\mu_{\mathrm{DA}},\mu_{\mathrm{ACh}},\mu_{\mathrm{NE}},\mu_{5HT},\mu_{\mathrm{OP}})$
are coupled to the the effective support of the routing matrix $\mathcal{R}_\controller$. Finer-grained valuation makes the modulation of salience scores $s_i$ more specific. Routing mass is concentrated on fewer subsystems, leaving less symbolic and geometric capacity for unconstrained deliberation. This tradeoff is structural for any closed system whose valuation and deliberation use the same
routing infrastructure.

\textit{Metalearning and the evolutionary timescale.} At the timescale of individual learning, the metacognitive layer of
the Metacognition section adjusts routing parameters, reliability estimates, and norm weights in response to prediction errors within a lifetime. A complete account of the value system would also need a third temporal tier in which selection across evolutionary time shapes the architecture of the neuromodulatory pathways, the functional form of the
homeostatic drive $f_h$, the prior distribution over world models, and the initial norm
weights $w_k^{(0)}$.
This gives the biological origin of primary values \citep{pugh1977}. Drives and affective
responses appear as given within an individual lifetime, but arise from prolonged optimization of reproductive fitness across generations. The synthetic counterpart is a hyperprior over the parametric structure of the valuative module. A process analogous to evolutionary search determines structural parameters of Specifications~\ref{bp:valuative}-\ref{bp:reliability}.

\textit{Latent World Model.} Other architectures can also serve as candidates for the latent world model. JEPA \citep{assranetal2023}, for example, is admissible here when its encoders and predictor are treated as bounded Lipschitz latent maps on compact finite-dimensional state spaces and its prediction error is routed into reliability, valuative, memory, or interconnector dynamics. We can then use the prediction error as one of the reliability errors $\epsilon_i$. It feeds the metacognitive reliability update and valuative modulation, and may also feed the attention/awareness precision fields.

\textit{Effective field theory interpretation and future directions.}
We can view the architecture through an analogy with effective field theory. The symbolic field $\leftstate$ and the geometric field $\rightstate$ evolve as dynamical systems on distince bases: sequence
positions $\{1,\ldots,T\}$ and graph nodes $V$, respectively. Each isolated system obeys its own field equation (Specifications~\ref{bp:symbolic}
and~\ref{bp:geometric}). These equations play the role of free-field terms in the
effective action. The coupling kernel $\mathcal{K}$ in Specification~\ref{bp:interconnector}
acts as the interaction vertex. For each pair consisting of a sequence position and a graph node, this vertex furnishes the linear map by which the fields exchange information. The
Hilbert-Schmidt norm bound $\|\mathcal{K}\|_{\mathrm{HS}}\leq C_\mathcal{K}$ loosely resembles a finiteness condition: total coupling strength remains
finite regardless of how many modes interact. 
Similarly, the coupling bottleneck $d_R$ loosely resembles a scale cutoff, since all
information exchanged between the fields is compressed to dimension $d_R$. The
other specifications describe modulating fields that adjust effective coupling
strengths and field parameters according to the system's state. The Lipschitz
conditions bound how rapidly fields and couplings may vary, playing a role analogous to regularity conditions. The symmetry condition in
Specification~\ref{bp:geometric}(ii) is a kind of symmetry
consistency requirement. The inter-specification consistency conditions identified in the
Formal Closure section \ref{sec:formal_closure} ensure that the effective description of each module is
compatible with the adjacent modules. A natural direction for further development is the
minimal coupling. By that we mean the determination of the smallest kernel $\mathcal{K}$ (in the sense of smallest $C_\mathcal{K}$) that reproduces the cognitive phenomena attributed to the architecture. This is the analogy for the minimal coupling principle in gauge theories, and it points toward a more axiomatic reformulation in which the field equations are derived from symmetry and minimality conditions rather than assembled from independently motivated modules. Furthermore, the most promising extensions are to replace the current compact-Lipschitz-dissipative admissibility regime with structural stability principles that use the architecture's own mechanisms. Concretely, the aim should be to prove stability for the full adaptive system with state-dependent attention, routing, valuation, memory, reliability, and policy dynamics. Then the strong residual damping should be replaced by incremental passivity or compositional small-gain conditions. Valuation should also be able to regulate reentrant coupling itself and enforce stability margins. We would like to analyze delay-dependent regimes where reentry produces oscillations, synchronization, or Hopf bifurcations. Another direction would be to develop continuum and large-graph limits with convergence of attractors and to incorporate online learning as well as stochasticity through parameter dynamics or stochastic RFDEs. A final question  is whether, and to what extent, coercive dissipativity can replace artificial compact projections through absorbing-set and attractor results.

\appendix

\section{Symbolic Module}

The symbolic module must maintain a dense sequential workspace adequate for linguistic, rule-based, counterfactual, and justificatory representations, while remaining open to grounding constraints from non-symbolic systems. We verify these requirements through the conditions below.

The symbolic module maintains the field $\leftstate\in\mathbb{R}^{T\times d_L}$. This field
represents utterances, plans, hypotheses, rules, obligations, and narrative summaries. Through serial abstraction, the field converts distributed system state into propositions available for comparison, revision, explanation, and memory. The module can receive four classes of input:
broadcast from the controller, geometric grounding packets from the interconnector,
precision modulation from the valuative system, and norm constraints from the executive
state. We therefore write its update as
\[
\leftstate^+ = \mathcal T_L\bigl(\leftstate,\;B_\controller,\;C_{R\to L},\;Q_L,\;
\frontal,\;\memory\bigr).
\]
The precision field $Q_L\in\mathbb{R}^{T\times T}_{>0}$ modulates attention over symbolic
positions. This yields a positive field over token pairs, increased by precision demand and
decreased by uncertainty. 

\begin{blueprint}[Symbolic update operator and precision field]\label{bp:symbolic}
The symbolic update operator $\mathcal{T}_L$ and precision field $Q_L$ belong to the
admissible class $\mathfrak{C}_L$ if:
\begin{enumerate}[label=(\roman*)]
\item \emph{Regularity.} $\mathcal{T}_L$ is Lipschitz continuous in all arguments with
  Lipschitz constant $\mathrm{Lip}_L$.
\item \emph{Bounded range.} $\|\leftstate^+\|_F\leq R_L$ for all inputs in the state domain.
\item \emph{Logit sensitivity.} For the softmax attention map with pre-softmax logit
  $q_{\ell s}$ and attention probability $p_{\ell s}=\mathrm{softmax}(q_{\ell\cdot})_s$,
  the derivative $\partial p_{\ell s}/\partial q_{\ell s}=p_{\ell s}(1-p_{\ell s})>0$
  confirms that increasing $Q_L(\ell,s)$ (which enters $q_{\ell s}$) strictly increases
  $p_{\ell s}$ within its row.
  Implementations using sparsemax satisfy weak monotonicity; strict increase holds on
  active coordinates while the active set is fixed and has cardinality greater than one.
\item \emph{Grounding sensitivity.} $\mathcal{T}_L$ depends non-trivially on $C_{R\to L}$:
  there exist inputs for which varying $C_{R\to L}$ changes the output of $\mathcal{T}_L$.
\item \emph{Positivity of $Q_L$.} $Q_L(t,s)>0$ for all token pairs $(t,s)$.
\item \emph{Monotonicity of $Q_L$.} $Q_L$ is monotone non-decreasing in the precision
  demand signal $\mu_{\mathrm{prec}}$ drawn from the valuative system.
\item \emph{Regularity of $Q_L$.} $Q_L$ is Lipschitz continuous in $\mu_{\mathrm{prec}}$
  and in the base logit field $q_L$.
\item \emph{Bounded positive parameterization.} Any admissible implementation produces
  $Q_L\in[\varepsilon_Q,R_Q]^{T\times T}$ for fixed architecture constants
  $0<\varepsilon_Q\leq R_Q$. A standard construction is
  $Q_L(\ell,s)=\varepsilon_Q+(R_Q-\varepsilon_Q)\sigma(\tilde{q}_L(\ell,s))$,
  where $\sigma$ is the logistic function and $\tilde{q}_L$ is any real-valued logit field.
\item \emph{Dissipativity.} The instantaneous reaction term $F_L:=\mathcal T_L^{(0)}-\mathrm{Id}$ is one-sided Lipschitz with constant
  $-\nu_L<0$: for all $u,v$ in the state domain,
  \[
  \langle F_L(u)-F_L(v),\,u-v\rangle_F\;\leq\;-\nu_L\|u-v\|_F^2.
  \]
  \begin{sloppypar}The stability theorem makes use of $\mu_L$ for the dissipativity constant of the combined principal operator $-\Delta_{G_L}(Q_L)+F_L$, where the graph-Laplacian contributes
  to $\mu_L$ alongside the reaction term.\end{sloppypar}
\end{enumerate}
\end{blueprint}

\textit{Existence argument.} Choose as an element of $\mathfrak{C}_L$ the stack of
$L_{\mathrm{sym}}$ transformer blocks, cf. \citep{vaswani2017attention}, with additional normalization. One step is layer normalization \citep{ba2016layer}, applied after each attention sublayer and feed-forward sublayer.
Condition~(i) is then fulfilled, since each relevant sublayer operation is Lipschitz on bounded inputs. This includes scaled dot-product attention, feed-forward network, and layer normalization. Their composition over $L_{\mathrm{sym}}$ layers is Lipschitz by the chain rule, with constant $\mathrm{Lip}_L$ depending on depth, embedding dimension, and the spectral norms of all weight matrices (query, key,
value, and feed-forward projections). Bounded spectral norms remain a standing assumption for this claim. In practice, spectral normalization \citep{miyato2018spectral} or the radial projection already appended in condition~(ii) serves this purpose. Layer normalization with damping parameter $\varepsilon > 0$, $\mathrm{LN}_\varepsilon(x) = \gamma \odot(x-\bar x) / \sqrt{\mathrm{Var}(x) + \varepsilon} + \beta$,
has Lipschitz constant bounded by $2\|\gamma\|_\infty / \sqrt{\varepsilon}$. For $\varepsilon > 0$
fixed this constant is finite, and the overall Lipschitz constant $\mathrm{Lip}_L$ of the composed
network is bounded accordingly.
To fulfill condition~(ii), we notice first that residual connections accumulate across layers and layer normalization alone fails to bound the output norm of a standard residual transformer. To verify condition~(ii) unconditionally, the architecture appends a final projection
$\leftstate\leftarrow\Pi_{B_{R_L}}(\leftstate)$ onto the Frobenius ball of radius $R_L$
after the last layer. This satisfies Specification~\ref{bp:symbolic}(iv) by construction
and is Lipschitz with constant one. For condition~(iii), note that $Q_L$ enters the attention logit as a positive additive term. Increasing $Q_L(\ell,s)$ raises the pre-softmax logit of the corresponding position pair, increasing the attention weight within that row. Condition~(iv): $C_{R\to L}$ is an additive contribution to the initial hidden state $\leftstate^{(0)}$; any non-zero grounding packet shifts $\leftstate^{(0)}$ and propagates through all subsequent layers. Conditions~(v)-(viii): use the logistic construction from condition~(viii), 
\[
Q_L(\ell,s)=\varepsilon_Q+(R_Q-\varepsilon_Q)\,\sigma(\tilde{q}_L(\ell,s)),
\]
with $\tilde{q}_L(\ell,s) = a_L(\ell,s) + \mu_{\mathrm{prec}}\, b_L(\ell,s)$ and $b_L \ge 0$ pointwise.
For bounded range, note that $\sigma$ maps $\mathbb{R}$ into $(0,1)$, so $Q_L \in (\varepsilon_Q, R_Q)\subset[\varepsilon_Q,R_Q]$ by construction.
We have positivity by $Q_L > \varepsilon_Q>0$.
Monotonicity in $\mu_{\mathrm{prec}}$: $\partial Q_L / \partial \mu_{\mathrm{prec}} = (R_Q-\varepsilon_Q)\,\sigma'(\tilde{q}_L)\,b_L\ge0$, with strict increase whenever $b_L > 0$.
Lipschitz continuity: $\sigma$ is $\tfrac{1}{4}$-Lipschitz, $\tilde{q}_L$ is linear in
$(\mu_{\mathrm{prec}},a_L,b_L)$, and the composition is Lipschitz on any compact parameter domain.

\textit{Dissipativity construction.} The constant field is an eigenvector of the graph Laplacian with eigenvalue zero. Since diffusion alone cannot contract the constant mode,
strict dissipativity of the symbolic block must come from the reaction term.
A sufficient construction is the \emph{residual damped transformer}:
\[
F_L(\leftstate,\ldots)\;=\;-\alpha\leftstate+\Phi_L(\leftstate,\ldots),
\qquad\mathrm{Lip}(\Phi_L)<\alpha,
\]
which yields
\begin{align*}
\langle F_L(H)-F_L(H'),H-H'\rangle_F &\leq(-\alpha+\mathrm{Lip}(\Phi_L))\|H-H'\|_F^2 \leq-\mu_L\|H-H'\|_F^2
\end{align*}
with $\mu_L=\alpha-\mathrm{Lip}(\Phi_L)>0$.
A standard transformer block with spectral-norm-bounded weights and a residual
linear damping term $-\alpha H$ satisfies this when
$\mathrm{Lip}(\Phi_L)\leq\alpha-\mu_L$ for some $\mu_L>0$.

\textit{Spectral gap of the precision field.} Under the standing assumption that $\GL$ is
connected and $Q_L\geq\varepsilon_Q>0$ (condition~(viii)), the symmetric part of
$\Delta_{\GL}(Q_L)$ has spectral gap at least $\varepsilon_Q\lambda_2(\GL)>0$, where
$\lambda_2(\GL)$ is the second eigenvalue of the symmetrized Laplacian (valid for
undirected graphs, compare the description in Section~\ref{sec:formal_closure}).
Higher $Q_L$ pointwise yields faster mixing of $\leftstate$ toward its quasi-stationary
configuration. The monotonicity condition~(vi) expresses this at the level of the
modulating signal. We need to explain next how to ensure the lower bound $Q_L\geq\varepsilon_Q>0$. Note that Softmax attention has strictly positive entries, but fails to deliver an input-uniform lower bound: attention weights can be exponentially small, and for long sequences even uniform attention scales like $1/T$. Sparse attention can set entries exactly to zero. Moreover, transformer attention is generally directed, whereas the spectral-gap estimate above is a statement about a
symmetric conductance Laplacian. In the present framework, $Q_L$ is therefore a
precision/conductance field. A realistic way to obtain the required gap is by adding a connected
residual backbone or floor,
\[
Q_L(Z)=Q_{\mathrm{base}}+\widetilde Q_L(Z),
\]
\[
Q_{\mathrm{base}}\geq\varepsilon_Q \ \text{on a fixed connected graph }G_0,
\]
\[
\widetilde Q_L(Z)\geq0,
\]
or equivalently by using a floored symmetric conductance field derived from attention.
Then
\[
\lambda_2(\Delta_{\GL}(Q_L))
\geq
\lambda_2(\Delta_{G_0}(Q_{\mathrm{base}}))>0,
\]
independently of whether the attention component becomes sharp, sparse, or directed. 

\textit{Coarse-graining.} The coarse-graining map $\Phi_L$ extracts the abstract classes
$\mathcal{A}_L$ the dissipative operator on $L^2(\GL,\mathbb{R}^{d_L})$ with dissipativity
constant $\mu_L>0$ and $\mathcal{A}_{Q_L}$ the positive precision fields whose symmetric
diffusion conductance has a uniform positive floor on a fixed connected support.

\section{Geometric Module}

\label{app:geo}

The rationale for the geometric module is that it should endow the system with a structured world. The detection of occurrences in the environment, e.g. that a plan is impossible, that a path is blocked, that an action threatens another agent, or that a safe route exists is the main function of this module. The corresponding representations are local, relational and preserve structure under transformation. The geometric module therefore functions as the system's spatial and affordance based contact with its environment. The geometric module should represent the world as a sparse equivariant relational structure and the awareness field selects objects, agents, affordances, and transformations relevant to action. 

This module maintains the field
$\rightstate=\{e_i\in\mathbb{R}^{d_R},\,g_{ij}\in\mathrm{SE}(d)\}_{(i,j)\in E}$. Its
state is a sparse structure over different kinds of vertices. These can include objects, body parts, regions, agents, paths, hazards, and affordances \citep{spelke2007core,locatello2020objectcentric}. The edges carry relations such as proximity, containment, occlusion, reachability, support, collision risk, and causal influence \citep{scholkopf2021causal}. The transformations between local frames can be represented via groupoid morphisms \citep{brown1987groupoids}. In our setting here, the rigid-frame transforms $g_{ij}\in\mathrm{SE}(d)$ are treated as fixed structural edge labels set at
initialization, and the node embeddings $e_i$ are dynamic state variables. The sparse awareness field $W_R \in [0,1]^{|E|}$ selects which relations should be propagated, exported, or used for action. The module's update has the following form
\[
\rightstate^+ = \mathcal G_R\bigl(\rightstate, B_\controller, C_{L\to R}, \valuative, 
\frontal, W_R, \memory\bigr).
\]
The field $C_{L\to R}$ injects symbolic goals, labels, hypotheses, and norms as constraints
on graph propagation.

\begin{blueprint}[Geometric update operator and awareness field]\label{bp:geometric}
The geometric update operator $\mathcal{G}_R$ and awareness field $W_R$ belong to the
admissible class $\mathfrak{C}_R$ if:
\begin{enumerate}[label=(\roman*)]
\item \emph{Regularity.} $\mathcal{G}_R$ is Lipschitz continuous in all arguments with
  constant $\mathrm{Lip}_R$.
\item \emph{Symmetry.} $\mathcal{G}_R$ is permutation-equivariant under $\mathrm{Aut}(G)$
  (reindexing nodes reindexes outputs correspondingly) and $\mathrm{SE}(d)$-invariant in the
  scalar feature sector. The group
  $\mathrm{SE}(d)$ acts on the fixed edge labels $g_{ij}$, and the update is invariant to
  that action because it enters only through $\mathrm{SE}(d)$-invariant quantities such as
  $\|r_{ij}\|^2$.
\item \emph{Sparsity preservation.} $\mathcal{G}_R$ depends on $\rightstate$ through
  $W_R$-weighted aggregations over the active support $\{(i,j):W_{R,ij}>0\}$; edges outside
  the support do not contribute to the update.
\item \emph{Bounded range.} $\|\rightstate^+\|\leq R_R$ for all nodes $i$ and all inputs in the
  state domain.
\item \emph{Simplex membership of $W_R$.} For each target node $j$, the incoming
  awareness weights satisfy
  \[
  \sum_{i:(i,j)\in E}W_{R,ij}=1, \quad W_{R,ij}\geq 0;
  \]
  $W_R$ lies in the product of simplices $\prod_j\Delta^{k_j}$, where $k_j$ is the
  in-degree of node $j$.
\item \emph{Regularity of $W_R$.} $W_R$ is Lipschitz continuous in the awareness logit
  field $\omega\in\mathbb{R}^{|E|}$, which is itself a Lipschitz function of $\rightstate$
  and the novelty signal $\mu_{\mathrm{nov}}$ from the valuative system.
\item \emph{Sparsity bound of $W_R$.} $|\mathrm{supp}(W_R)|\leq K_{\mathrm{sp}}$ for some
  fixed $K_{\mathrm{sp}}$ independent of $|E|$. This bound is enabled by a bounded-in-degree
  assumption on the graph: $\sup_j k_j\leq K_{\max}$ for some fixed $K_{\max}$, so that
  $K_{\mathrm{sp}}\leq |V|\cdot K_{\max}$.
\item \emph{Dissipativity.} The instantaneous reaction term
$F_R := \mathcal G_R^{(0)} - \mathrm{Id}$ is one-sided Lipschitz with constant
  $-\nu_R < 0$: for all $u,v$ in the state domain,
  \[
  \langle F_R(u)-F_R(v),\,u-v\rangle \ \leq \ -\nu_R \|u-v\|^2.
  \]
  This condition is compatible with $\mathrm{SE}(d)$-invariance of the scalar feature sector.
  The stability theorem makes use of $\mu_R$ for the dissipativity constant of the combined
  principal operator $-\Delta_{G_R}(W_R)+F_R$. The graph-Laplacian diffusion contributes to
  $\mu_R$ alongside the reaction term.
\end{enumerate}
\end{blueprint}

\textit{Existence argument.} A concrete element of $\mathfrak{C}_R$ is the scalar-feature
sector of an EGNN-style architecture similar to \citet{satorras2021egnn}.
Each node $i\in V$ carries an embedding $e_i\in\mathbb{R}^{d_R}$ and each directed edge
$(i,j)\in E$ carries a rigid-frame transform $g_{ij}=(R_{ij},t_{ij})\in\mathrm{SE}(d)$
used as a fixed edge attribute, with relative position vector $r_{ij}:=t_{ij}\in\mathbb{R}^d$
extracted from its translational component.
Under a global rigid motion $(R,t)\in\mathrm{SE}(d)$ applied to node positions,
$r_{ij}\mapsto R\,r_{ij}$, so $\|r_{ij}\|^2$ is invariant.
In this instantiation the position-coordinate dependence is captured
entirely by $r_{ij}$, which enters only through the $\mathrm{SE}(d)$-invariant distance
$\|r_{ij}\|^2$. The scalar features $e_i$ are therefore $\mathrm{SE}(d)$-invariant.
The messages are computed as $m_{ij}=\phi_e(e_i,e_j,\|r_{ij}\|^2,a_{ij})$ where $a_{ij}$ encodes additional edge attributes and $\mathrm{SE}(d)$-invariance holds because $m_{ij}$ depends on $g_{ij}$ only through $\mathrm{SE}(d)$-invariant quantities (distances and inner products). The permutation-equivariance holds because the node update $e_j^+=\phi_h(e_j,\sum_{i\in S^*}W_{R,ij}m_{ij})$
reindexes correctly under any $\sigma\in\mathrm{Aut}(G)$. This
architecture satisfies the scalar-feature part of condition~(ii), following the invariant
message construction, cf. \citet{satorras2021egnn}.
Condition~(i): the message and update functions $\phi_e,\phi_h$ are MLPs with layer
normalization ($\varepsilon>0$, Lipschitz constant $2\|\gamma\|_\infty/\sqrt{\varepsilon}$)
and bounded range, giving a finite Lipschitz constant $\mathrm{Lip}_R$ since each factor is Lipschitz.
Condition~(iii) holds by
definition of the awareness-weighted message $m_j=\sum_{i\in S^*}W_{R,ij}m_{ij}$,
where only active-support edges contribute. Condition~(iv): In order to obtain $\|\rightstate^+\|\leq R_R$ when residual connections or subsequent affine maps are present, append a final projection $\rightstate\leftarrow\Pi_{B_{R_R}}(e_j)$ (clipping to the Euclidean ball of radius $R_R$) after the last EGNN layer and this satisfies Specification~\ref{bp:geometric}(iv) by
construction, i.e. is Lipschitz with constant one.
Conditions~(v)-(vii) are satisfied by the sparsemax projection \citep{martins2016softmax}
$W_{R,ij}=\max(\omega_{ij}-\tau,0)$, where $\tau$ is the threshold enforcing
$\sum_{i\in S^*}W_{R,ij}=1$: simplex membership holds by construction. The Lipschitz continuity
holds because the sparsemax map is the Euclidean projection onto the simplex, which is
non-expansive. The sparsity bound holds because $|S^*|$ is bounded by the in-degree $k_j\leq K_{\max}$, and summing over all nodes yields $|\mathrm{supp}(W_R)| \leq |V|\cdot K_{\max} =: K_{\mathrm{sp}}$.
Softmax has full support over all edges, so it satisfies the sparsity bound of
condition~(vii) only if $K_{\mathrm{sp}}=|E|$ (i.e.\ the bound is vacuous) or the total
edge count is bounded. Strict sparsity requires a different approach.
Top-$K$ normalization satisfies conditions~(v) and~(vii) but is discontinuous at ties,
placing it outside the admissible class for the Lipschitz results. Other Lipschitz
projections onto the simplex with bounded support (e.g.\ sparsemax with bounded in-degree)
are admissible within $\mathfrak{C}_R$ provided they satisfy conditions~(v)-(vii).

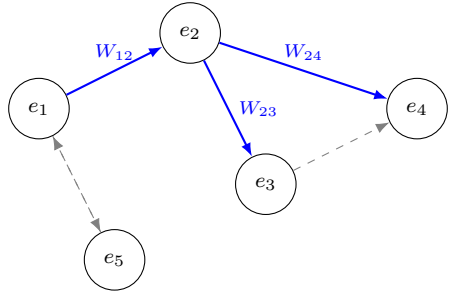
\begin{figure*}[t]
\centering
\resizebox{0.5\linewidth}{!}{%
\begin{tikzpicture}[
  obj/.style={circle, draw, minimum size=8mm, font=\scriptsize},
  edge/.style={-{Latex[length=1.8mm]}, thick},
  aware/.style={-{Latex[length=1.8mm]}, thick, blue},
  dim/.style={-{Latex[length=1.8mm]}, dashed, gray}
]
\node[obj] (a) at (0,0) {$e_1$};
\node[obj] (b) at (2,1) {$e_2$};
\node[obj] (c) at (3,-1) {$e_3$};
\node[obj] (d) at (5,0) {$e_4$};
\node[obj] (e) at (1,-2) {$e_5$};

\draw[aware] (a) -- node[above,font=\tiny]{$W_{12}$} (b);
\draw[aware] (b) -- node[above,font=\tiny]{$W_{24}$} (d);
\draw[dim] (c) -- (d);
\draw[dim] (e) -- (a);
\draw[aware] (b) -- node[right,font=\tiny]{$W_{23}$} (c);
\draw[dim] (a) -- (e);
\end{tikzpicture}}
\caption{A sparse right-side representation. Blue arrows carry active awareness weight;
grey arrows are suppressed. Awareness is allocated to relations that affect action, risk,
body state, or other agents.}
\end{figure*}

\textit{Dissipativity construction.} As with the symbolic block, the graph-Laplacian
diffusion does not contract constant-field modes. Therefore, the strict dissipativity of the geometric
block must come from the reaction term. A sufficient construction is:
\[
F_R(\rightstate,\ldots) = -\alpha_R\rightstate + \Phi_R(\rightstate,\ldots), \ \mathrm{Lip}(\Phi_R)<\alpha_R,
\]
which yields $\langle F_R(e)-F_R(e'),e-e'\rangle \leq -\mu_R \|e-e'\|^2$ with $\mu_R = \alpha_R-\mathrm{Lip}(\Phi_R)>0$. The EGNN existence argument establishes Lipschitz and bounded range. When this is augmented with a residual linear damping $-\alpha_R e$ then dissipativity condition~(viii) is satisfied.

\textit{Spectral gap of the awareness field.}
The simplex condition on \(W_R\) furnishes nonnegative normalized awareness weights. In the undirected case, the symmetric awareness conductance is obtained from \(W_R\), for instance by
\[
\overline W_{R,ij} = \tfrac12(W_{R,ij} + W_{R,ji}),
\]
and the relevant Laplacian is \(\Delta_{\GR}(\overline W_R)\). A positive gap holds
whenever the symmetrized active support
\[
E_R(Z)=\{\{i,j\}:\overline W_{R,ij}(Z)>0\}
\]
is connected. Equivalently,
\[
\lambda_2(\Delta_{\GR}(\overline W_R)) > 0
\]
is an additional connectivity condition on the awareness graph. For sparse awareness fields this means that \(W_R\) may be highly selective, provided its symmetrized active support
remains connected, or provided a fixed sparse residual backbone furnishes the gap.
In the latter case one may write
\begin{align*}
& \overline W_R(Z) = W_{\mathrm{base}}+\widetilde W_R(Z), \
W_{\mathrm{base}} \geq \varepsilon_R>0 \\
&\text{on a fixed connected sparse graph} G_{\mathrm{base}}, \
\widetilde W_R(Z) \geq 0.
\end{align*}
This yields the uniform lower bound
\[
\lambda_2(\Delta_{\GR}(\overline W_R(Z)))
\geq
\lambda_2(\Delta_{G_{\mathrm{base}}}(W_{\mathrm{base}}))>0.
\]
More generally, for a sparse connected active graph, Cheeger-type estimates relate the
spectral gap to the conductance \(h(\GR,\overline W_R)\), for instance
\[
\lambda_2(\Delta_{\GR}(\overline W_R))
\gtrsim h(\GR,\overline W_R)^2,
\]
up to the normalization convention for the graph Laplacian. The gap is obtained from the connected
conductance support or by an explicitly assumed residual backbone.

\begin{sloppypar}
\textit{Coarse-graining.} The coarse-graining map $\Phi_R$ extracts the abstract classes
$\mathcal{A}_R$ (dissipative operator on $L^2(\GR,\mathbb{R}^{d_R})$ with
$\mathrm{SE}(d)$-in\-var\-iance in the scalar feature sector,
$\mathrm{Aut}(\GR)$-equiv\-ar\-iance, dissipativity constant $\mu_R>0$)
and $\mathcal{A}_{W_R}$ (edge weight function $w_R\colon E_R\to\mathbb{R}_{\geq 0}$ on
simplices with connected support). Microscopic details (e.g.\ EGNN layer count,
message-passing readout, MLP width) are not relevant to $\mathcal{A}_R$;
see Section~\ref{sec:coarsegrained}.
\end{sloppypar}

\section{Interconnector}

\label{app:interconnector}

The interconnector should act as a delayed, gated translation operator between heterogeneous latent spaces. It is supposed to preserve specialization, but at the same time allow coordinated behavior. The function of the interconnector can perhaps best be described as disciplined exchange. The left side sends hypotheses such as "test whether this path is safe." The right side sends constraints such as
"this action would collide with another agent." The healthy state of this system is that of partial
synchronization by which is meant that enough coupling is available to coordinate, and enough separation to preserve distinct computations.

This module is supposed to play a role roughly analogous of the corpus callosum. It is defined as a pair of translation maps with routing gains and delays. The dense symbolic content is turned into graph constraints, labels, goals, and queries by the interconnector. The sparse geometric content is transformed into symbolic packets, contradictions, affordance summaries, and risk statements. We write the interconnector as
\begin{align*}
C_{R\to L}&=g_{R\to L}\,\Phi_{R\to L}\!\bigl(\rightstate(t-\tau_{R\to L})\bigr),\\
C_{L\to R}&=g_{L\to R}\,\Phi_{L\to R}\!\bigl(\leftstate(t-\tau_{L\to R})\bigr),
\end{align*}
where $\tau_{R\to L}$ and $\tau_{L\to R}$ are non-negative delay parameters and
$g_{R\to L},g_{L\to R}\in[0,1]$ are scalar gates controlled by the valuative and executive
state. The delays allow phase, anticipation, and synchronization effects. These translation
maps are calibrated by experience.

\begin{blueprint}[Coupling kernel and interconnector]\label{bp:interconnector}
The interconnector is characterized by a coupling kernel
$\mathcal{K}:\{1,\ldots,T\}\times V\to\mathbb{R}^{d_L\times d_R}$ that assigns to each
(sequence position, graph node) pair a linear map from the geometric to the symbolic
embedding space. The interconnector belongs to the admissible class
$\mathfrak{C}_\mathcal{I}$ if:
\begin{sloppypar}
\begin{enumerate}[label=(\roman*)]
\item \emph{Bounded coupling.}
  \[
  \|\mathcal{K}\|_{\mathrm{HS}}^2 := \sum_{\ell=1}^{T}\sum_{i\in V}\|\mathcal{K}(\ell,i)\|_F^2
  \leq C_\mathcal{K}^2.
  \]
  If $\mathcal{K}$ depends on $Z$, it is Lipschitz in $Z$ with constant $L_\mathcal{K}$.
\item \emph{Right-to-left signal.}
  $C_{R\to L,\ell} = \sum_{i\in V}\alpha_{\ell i}(Z)\,\mathcal{K}(\ell,i)\,e_i$,
  where $\alpha(\cdot,Z)\in\Delta^{|V|-1}$ for each $\ell$ and is Lipschitz in $Z$.
\item \emph{Left-to-right signal.}
  \[
  C_{L\to R,i} = \sum_{\ell=1}^{T}\beta_{i\ell}(Z)\,\mathcal{K}(\ell,i)^{\!\top} H_{L,\ell},
  \]
  where $H_{L,\ell}\in\mathbb{R}^{d_L}$ is the $\ell$-th token embedding of $\leftstate$,
  and $\beta(\cdot,Z)\in\Delta^{T-1}$ for each $i$ and is Lipschitz in $Z$.
\end{enumerate}
\end{sloppypar}
The dimension $d_R$ is the geometric embedding dimension already introduced for $\rightstate$;
it acts as the \emph{coupling bottleneck}: all geometric information reaching the symbolic
field, and all symbolic information reaching the geometric field, passes through this
dimension. Gate values $g_{R\to L}$ and
$g_{L\to R}$ are entries of $\mathcal{R}_\controller\in[0,1]^{n_s\times n_s}$; their
boundedness is guaranteed by Specification~\ref{bp:controller}(i) and requires no separate
gate condition here. The delay scalars $\tau_{R\to L}\ge0$ and $\tau_{L\to R}\ge0$ are
fixed architecture constants.
\end{blueprint}

\begin{sloppypar}
\textit{Existence argument.} The simplest element of $\mathfrak{C}_\mathcal{I}$ is the
state-independent constant kernel $\mathcal{K}(\ell,i) = W_{RL}\in\mathbb{R}^{d_L\times d_R}$
for all $(\ell,i)$. This kernel has Hilbert-Schmidt norm $\sqrt{T|V|}\,\|W_{RL}\|_F$, bounded by
the architectural constraint $\|W_{RL}\|_F\leq C_{RL} : =C_\mathcal{K}/\sqrt{T|V|}$
which is enforced by projecting $W_{RL}$ onto the Frobenius-norm ball after each update. With this scaling the HS norm satisfies
\[
\|\mathcal{K}\|_{\mathrm{HS}} = \sqrt{T|V|}\, \|W_{RL}\|_F\leq C_\mathcal{K}, 
\]
and the coupling strength per interaction pair scales as $1 / \sqrt{T|V|}$, so the total coupling
remains bounded as the number of interacting modes grows. When the simplex weights $\alpha_{\ell i}$ are computed by any Lipschitz function of $Z$ (e.g. a sparsemax over dot products between token and node features), the right-to-left signal is $C_{R\to L,\ell} = W_{RL}\,\bar{e}_\ell$ where
$\bar{e}_\ell = \sum_i\alpha_{\ell i}\,e_i\in\mathbb{R}^{d_R}$. This recovers the
bounded linear map as a special case. A more informative instance is the
cross-attention kernel given by $\mathcal{K}(\ell,i)=W_V$ (a value projection, still
constant in $(\ell,i)$), with 
\[
\alpha_{\ell i} = \mathrm{sparsemax}\left(q_\ell^\top \frac{k_i}{\sqrt{d_\mathcal{K}}}\right)_i \ \text{where} \ q_\ell=W_Q H_{L,\ell}
\]
and $k_i=W_K e_i$ are Lipschitz functions of $Z$. The Lipschitz constants depend on the spectral norms $\|W_Q\|$, $\|W_K\|$, $\|W_V\|$, which are
assumed to be bounded, i.e. $\|W_Q\|,\|W_K\|,\|W_V\|\leq C_W$ for some fixed architecture
constant $C_W$. The coupling weights are now sequence- and node-dependent, but the kernel
matrix remains shared.
A fully $(\ell,i)$-dependent kernel $\mathcal{K}(\ell,i,Z) = f(H_{L, \ell}, e_i)$ for some
Lipschitz $f$ is the most general admissible instance. The HS norm of that kernel is bounded when $f$
has bounded Lipschitz constant and the state domain is compact. We include an overview of these different kernel types below.
\end{sloppypar}

The delay parameters $\tau_{R\to L}$ and $\tau_{L\to R}$ are treated here as externally
specified schedule constants which means that they are initialized from structural priors and may be
adjusted between episodes, but are fixed during any single rollout. If online adaptation of delays is desired, the update rule and bounds must be included in an augmented Markov state, which is outside the scope of the present analysis. The interconnector gates are specific entries of the routing matrix $\mathcal{R}_\controller$: $g_{R\to L}=\mathcal{R}_{\controller,R\to L}$ and
$g_{L\to R}=\mathcal{R}_{\controller,L\to R}$, where the entries are already in $[0,1]$
by the routing specification below.

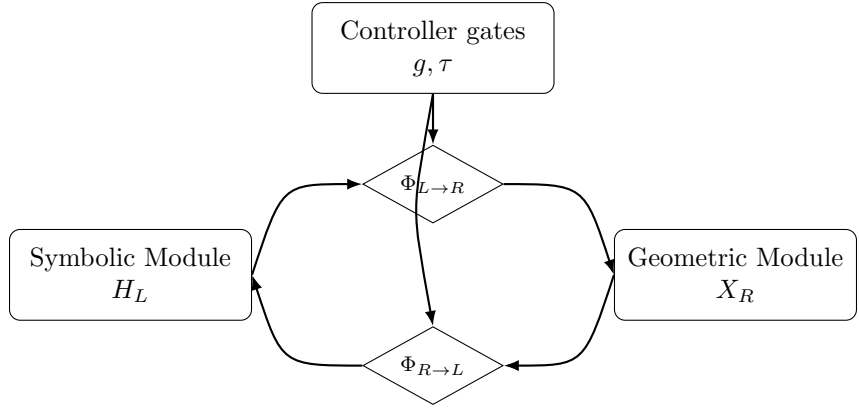
\begin{figure*}[t]
\centering
\resizebox{0.75\linewidth}{!}{%
\begin{tikzpicture}[
  box/.style={draw, rounded corners, align=center, minimum width=32mm, minimum height=12mm,
    font=\small},
  op/.style={draw, diamond, aspect=1.8, align=center, font=\scriptsize, inner sep=1mm},
  arrow/.style={-{Latex[length=2mm]}, thick}
]
\node[box] (left) at (0,0) {Symbolic Module\\$\leftstate$};
\node[box] (right) at (8,0) {Geometric Module\\$\rightstate$};
\node[op] (tlr) at (4,1.2) {$\Phi_{L\to R}$};
\node[op] (trl) at (4,-1.2) {$\Phi_{R\to L}$};
\node[box] (ctrl) at (4,3) {Controller gates\\$g,\tau$};
\draw[arrow] (left.east) .. controls (2,1.2) .. (tlr.west);
\draw[arrow] (tlr.east) .. controls (6,1.2) .. (right.west);
\draw[arrow] (right.west) .. controls (6,-1.2) .. (trl.east);
\draw[arrow] (trl.west) .. controls (2,-1.2) .. (left.east);
\draw[arrow] (ctrl.south) -- (tlr.north);
\draw[arrow] (ctrl.south) to[out=260, in=100, looseness=1.4] (trl.north);
\end{tikzpicture}}
\caption{The interconnector translates across heterogeneous latent spaces and is gated by
the controller. The translation operators $\Phi_{R\to L}$ and $\Phi_{L\to R}$ are bounded
maps; their linear form is a special instance of Specification~\ref{bp:interconnector}.}
\end{figure*}

\textit{Coarse-graining.} The coarse-graining map $\Phi_\Kcal$ extracts the abstract class
$\mathcal{A}_\Kcal$: Hilbert-Schmidt bipartite operators
$\Kcal\colon L^2(\GR,\mathbb{R}^{d_R})\to L^2(\GL,\mathbb{R}^{d_L})$ with
$\|\Kcal\|_{\mathrm{HS}}\leq C_\Kcal$.  The attention weights $\alpha_{\ell i}$, gating
functions, and concrete cross-attention implementation enter the Lipschitz RFDE class
through the maps $Z\mapsto \Kcal_\alpha(Z)$ and $Z\mapsto \Kcal_\beta^*(Z)$; the closed stability
class freezes them to fixed bounded operators; see Section~\ref{sec:coarsegrained}.

\subsection*{Kernel families and Hilbert-Schmidt scaling}\label{Kernel}

The Hilbert-Schmidt constraint $\|\Kcal(Z)\|_{\mathrm{HS}}\le C_\Kcal$ is an energy budget for
symbolic-geometric reentry. The budget interacts with system size in different ways
depending on the kernel family.

\textit{Constant shared kernel.} $\Kcal(\ell,i)=W_{RL}$ for all $(\ell,i)$.  Then
\[
\|\Kcal\|_{\mathrm{HS}}=\sqrt{T|V|}\,\|W_{RL}\|_F,
\]
so the budget requires $\|W_{RL}\|_F\le C_\Kcal/\sqrt{T|V|}$.  Dense coupling scales inversely with
the square root of the number of token-node pairs.

\textit{Attention-weighted translation kernel.}
\[
\Kcal_\alpha(\ell,i;Z )= \alpha_{\ell i}(Z)W_V,
\]
\begin{sloppypar}
with $\alpha_\ell(\cdot,Z) \in \Delta^{|V|-1}$. Since attention weights are
sim\-plex-valued, $\sum_i\alpha_{\ell i}^2\le1$ and
\end{sloppypar}
\[
\|\Kcal_\alpha(Z)\|_{\mathrm{HS}}^2\le T\|W_V\|_F^2,
\]
so the budget requires only $\|W_V\|_F\le C_\Kcal/\sqrt{T}$, removing the $\sqrt{|V|}$ penalty.

\textit{Low-rank structured kernel.} $\Kcal(\ell,i)=\sum_{r=1}^m a_r(\ell)b_r(i)A_r$
with $\sum_\ell a_r^2\le1$ and $\sum_i b_r^2\le1$.  Then
$\|\Kcal\|_{\mathrm{HS}}\le\sum_r\|A_r\|_F$, independent of $T$ and $|V|$.  Each rank channel
$r$ can encode a mode of cross-domain alignment (object identity, affordance, threat, task
relevance).

\textit{Gated mixture kernel.} $\Kcal(\ell,i;Z)=\sum_{r=1}^m g_r(Z)\Kcal_r(\ell,i)$
with $0\le g_r(Z)\le\bar g_r$.  Then $\|\Kcal(Z)\|_{\mathrm{HS}}\le\sum_r\bar g_r\|\Kcal_r\|_{\mathrm{HS}}$.
Executive and valuative variables can open and close coupling channels without violating the
energy budget. We note that state-dependent gates will contribute additional Lipschitz terms, which must be taken into account for a state-dependent stability margin.

\textit{Low rank gated attention.} The following variant combines gating, low-rank
structure, and attention:
\begin{equation}\label{eq:bestkernel}
  \Kcal(\ell,i;Z) = \sum_{r=1}^m g_r(Z)\,a_r(\ell,Z)\,b_r(i,Z)\,A_r.
\end{equation}
The norm bound $\|\Kcal(Z)\|_{\mathrm{HS}}\le C_\Kcal$ holds uniformly in $Z$.
This type of kernel can represent multiple reentrant channels and gives executive
and valuative variables a role as channel gates. For the global
stability theorem, the gates are frozen to their equilibrium values, making $\Kcal$ a fixed
Hilbert-Schmidt operator.
This recommendation presupposes a compute budget adequate for dynamic gating. The constant shared kernel is adequate for small $T$ and $|V|$.

\textit{High dimensional scaling.} Since dense unnormalized kernels have HS-norm scaling like \(\sqrt{d_Ld_R}\), or like \(\sqrt{T|V|d_Ld_R}\) for dense token-node coupling, they shrink the admissible small-gain margin for a reentry gain \(k\). Therefore, one should select low-rank normalized channels \(K_0 = USV^\top\) with controlled singular values, attention-weighted token-node selection, and bounded state-dependent gates.

\section{Controller}

Routing determines which content is transmitted between modules and in what proportion. The controller implements this as a routing field over a product of simplices; the admissibility conditions require the simplex to be positively invariant and the routing map to be Lipschitz-regular. The controller is also the site of arbitration. If the symbolic module proposes a plan and
the geometric module detects infeasibility, the controller increases right-to-left routing,
suppresses direct action, and requests replanning. If the valuative system signals urgent
harm, the controller broadens broadcast and shortens action latency. If the executive system
detects norm conflict, the controller routes other-agent geometry and relevant memories into
the symbolic module.

This module can be viewed as an analogue of the thalamus. In the present framework it serves as a dynamic routing system. Concretely, it receives exports from the symbolic module, geometric module, valuative system, executive system, and memory. It selects contents for broadcast, opens and closes channels, changes gain, coordinates timing, and determines which subsystem currently constrains action.

Let $Z_i$ be exported latent packets from subsystem $i$. The controller computes a sparse
routing matrix $\mathcal{R}_\controller$ and a Lipschitz sparse broadcast. Let
\[
w(Z):=\operatorname{sparsemax}(s(Z))\in\Delta^{n_s-1},
\
B_\controller:= \bigl(w_i(Z)Z_i\bigr)_{i=1}^{n_s},
\]
and
\[
M_{i\to j} = \mathcal{R}_{\controller,ij}\,T_{i\to j}(Z_i).
\]
The salience map $s(Z)$ is Lipschitz. Since sparsemax is non-expansive and all exported
packets lie in the compact state domain, the broadcast map $Z \mapsto B_\controller$ is
Lipschitz. 

\begin{blueprint}[Routing operator]\label{bp:controller}
The routing operator $K:Z\to\mathcal{R}_\controller$ belongs to the admissible class
$\mathfrak{C}_K$ if:
\begin{enumerate}[label=(\roman*)]
\item \emph{Row-stochasticity.} $\mathcal{R}_\controller\in[0,1]^{n_s\times n_s}$, with
  each row summing to one.
\item \emph{Sparsity.} Each row of $\mathcal{R}_\controller$ has support of cardinality at
  most $K_s$ for some fixed $K_s\leq n_s$.
\item \emph{Regularity.} $K$ is Lipschitz continuous in $Z$.
\end{enumerate}
\end{blueprint}

\textit{Existence argument.} For the Lipschitz theorem, set $K_s=n_s$.
Row-wise sparsemax \citep{martins2016softmax} applied to a score matrix
$S\in\mathbb{R}^{n_s\times n_s}$ satisfies all three conditions: row-stochasticity holds
by construction; sparsity condition~(ii) holds trivially with $K_s=n_s$; Lipschitz
continuity holds because sparsemax is a Euclidean projection onto the simplex and therefore
non-expansive with Lipschitz constant one. Row-wise softmax also satisfies conditions~(i)
and~(iii) with $K_s=n_s$.
The case therefore covers dense routing only: with $K_s=n_s$ every row of
$\mathcal{R}_\controller$ may have full support, making the sparsity condition vacuous.
The sparse-routing design principle is realized by the hard top-$K_s$
construction and approximately by Gumbel-Softmax, both described below. However, these fall outside
the Lipschitz conditions.

\textit{Admissibility.} Let us note first what is inadmissible. First, note that hard top-$K_s$ normalization with $K_s < n_s$ enforces strict sparsity but is discontinuous at score ties which therefore violates condition~(iii). Secondly, Gumbel-Softmax with any fixed positive temperature has full support over $n_s$ entries and therefore does not satisfy condition~(ii) exactly; it too is outside $\mathfrak{C}_K$. Both are viable implementations, but neither is covered by
the propositions. In the utilization of these constructions one must therefore independently verify any consistency properties claimed for the architecture. On the other hand, sparsemax is a Euclidean projection onto the simplex. Hence it is globally Lipschitz and piecewise affine; its active support can change across faces of the simplex, while the map itself remains continuous. The interconnector gate values $g_{R\to L}$ and $g_{L\to R}$ are specific entries of $\mathcal{R}_\controller$ and therefore lie in $[0,1]$ by condition~(i). 

\textit{Stability and the routing matrix.} The row-stochastic condition~(i) yields
$\|\mathcal{R}_\controller\|_\infty=1$, the natural operator norm for routing. The closed principal stability theorem absorbs only the fixed interfield operators $\Kcal$ and $\Kcal^*$ through the small gain margin. If routing is added directly to the principal stability block,
then its operator norm must be included in the same block small-gain matrix. For the row-stochastic admissible class used here, $\|\mathcal{R}_\controller\|_\infty=1$. A strictly row-substochastic routing rule belongs to a different row-subprobability state domain and would satisfy
$\|\mathcal{R}_\controller\|_\infty < 1$.

\begin{sloppypar}
The discussion above results in three corresponding regimes:
\begin{enumerate}[label=(\arabic*)]
\item Use softmax or another smooth positive-temperature map. The routing map is smooth,
  the RFDE vector field is classical, and derivative-based Lyapunov arguments apply
  directly.
\item Use sparsemax. The RFDE vector field remains Lipschitz and single-valued, so
  Picard well-posedness holds.  Variational stability arguments at support boundaries use
  the Clarke generalized Jacobian $\partial_C\mathcal{R}_\controller$.
\item Use hard top-$K_s$.  This requires a separate tie-margin, hybrid, or differential
  inclusion analysis and is outside the Lipschitz conditions stated here.
\end{enumerate}
\end{sloppypar}
\begin{sloppypar}
The deriv\-ative-based stability theorem uses regime~(1). Regime~(2) is admissible
for well-posedness and requires non\-smooth analysis for deriv\-ative-based stability.
\end{sloppypar}

\textit{Coarse-graining.} The coarse-graining map $\Phi_{R_\Theta}$ extracts the abstract class
$\mathcal{A}_{R_\Theta}$: row-stochastic weighted adjacency matrices on the subsystem graph
$G_C$, with $\|\mathcal{R}_\controller\|_\infty=1$ and Lipschitz state dependence.

\section{Valuative System}

We can view the valuative system or valence system as the central part of the framework. Here, the valences act as regulators. This regulation affects gain, priorities, plasticity and the impetus to perform actions. The function of the valuative system is the conversion of all homeostatic constraints, which consist of any deviaton from homeostasis that could for instance be rooted in prediction errors (in predictive coding based architectures), novel or unexpected occurrences, relief etc., into modulating signals.

This module could perhaps be viewed as an analogue of the hypothalamus, among other structures, which tracks viability variables and produces drive. The analogue of the frontal system regulates inhibition, planning, and norms. Both modules together then define the valuative organization of the system whose responsibility it is to signal the relevant states.

Let $h$ denote homeostatic deviation, $\varepsilon_{\mathrm{pred}}$ prediction error,
$n$ novelty, $r$ outcome feedback, and $\mu$ the neuromodulatory context. The map
$\mathcal{V}$ is decomposed as $\mathcal{V}=(\mathcal{V}_Y,\mathcal{V}_\mu)$, where
$\mathcal{V}_Y$ produces the updated valuative state and $\mathcal{V}_\mu$ produces the
neuromodulatory readout. We can write the update as
\begin{align*}
\valuative^+ &= \mathcal V_Y(\valuative,h,\varepsilon_{\mathrm{pred}}, n, r, \leftstate, \rightstate,\memory,\frontal),\\
\mu &= \mathcal{V}_\mu(\valuative)=(\mu_{\mathrm{DA}},\mu_{\mathrm{ACh}},\mu_{\mathrm{NE}},\mu_{5HT},\mu_{\mathrm{OP}}).
\end{align*}
Here $\mu_{\mathrm{DA}}$ is a value error (dopamine-like), $\mu_{\mathrm{ACh}}$ a
precision demand (acetycholine-like), $\mu_{\mathrm{NE}}$ an unexpected
uncertainty (norepinephrine-like), $\mu_{5HT}$ a temporal stabilization (serotonine-like), and $\mu_{\mathrm{OP}}$ relief (akin to opoid).

\begin{blueprint}[Valuative update and neuromodulation]\label{bp:valuative}
The valuative update operator $\mathcal{V}=(\mathcal{V}_Y,\mathcal{V}_\mu)$ and
neuromodulatory map belong to the admissible class $\mathfrak{C}_V$ if:
\begin{sloppypar}
\begin{enumerate}[label=(\roman*)]
\item \emph{Homeostatic dynamics.} The deviation variable $h$ obeys a leaky integrator:
  $h^+=h+\Delta t(-\kappa_h\,h+f_h(u))$, with $\kappa_h>0$ and $\|f_h(u)\|\leq B_u$ for
  all inputs $u$ in the perceptual domain.
\item \emph{Bounded modulation.} The neuromodulatory vector satisfies
  $\mu=(\mu_1,\ldots,\mu_5)\in(0,1)^5$.
\item \emph{Regularity.} $\mathcal{V}$ is Lipschitz continuous in all arguments.
\item \emph{Bounded valuative state.} $\|\mathcal{V}_Y(\cdot)\|\leq R_Y$ for all inputs
  in the state domain, where $R_Y$ is a finite constant depending on the architecture parameters.
\item \emph{Dissipativity of $G_Y$.} The forcing function $G_Y$ in
\[
\dot{\valuative}
=
-\kappa_Y\valuative+G_Y(\leftstate,\rightstate,u),
\qquad \kappa_Y>0,
\]
is bounded and Lipschitz on the admissible state domain. Holding
$(\leftstate,\rightstate,u)$ fixed, the valuative vector field is one-sided dissipative in
$\valuative$ with constant $\kappa_Y$. Set $\Phi_Y := \frac{G_Y}{\kappa_Y}$ and note that
$\dot{\valuative} = \kappa_Y(\Phi_Y - \valuative)$. The estimate is
the one stated in Theorem~\ref{thm:slowfast} under the corresponding hypothesis on $\dot \Phi_Y$ stated therein.
\end{enumerate}
\end{sloppypar}
\end{blueprint}

\begin{sloppypar}
\textit{Existence argument.} Condition~(i) is satisfied by the discrete leaky integrator
\[
h^{t+1}=(1-\kappa_h\Delta t)h^t+\Delta t\,f_h(u^t).
\]
Under the condition $0 < \Delta t < \kappa_h^{-1}$, we have that $0 < 1 - \kappa_h \Delta t < 1$, so the recursion
\[
\|h^{t+1}\| \leq (1-\kappa_h \Delta t) \|h^t\| + \Delta t B_u
\]
is contractive; every fixed point satisfies $\|h^*\|\leq B_u \kappa_h^{-1}$ which yields
\[
\|h^t\| \leq \max(\|h^0\|,B_u/\kappa_h) 
\]
for all $t\geq 0$. The continuous formula is 
\[
h(t)=e^{-\kappa_h t}h(0) + \int_0^t e^{-\kappa_h(t-s)} f_h(u(s))\,ds
\]
with the discrete estimate above.
Condition~(ii) is satisfied by any strictly monotone map into the open interval $(0,1)$. A sigmoid applied to a linear function of the state vector, $\mu_k=\sigma(\mathbf{w}_k^\top z)$, satisfies
conditions~(ii) and~(iii) for any bounded weight vector $\mathbf{w}_k$ and bounded input
domain. Other maps such as softsign, $\tanh$ shifted to $(0,1)$, or any other bounded
smooth activation are equally admissible. Condition~(iii) holds for all smooth bounded
activations by standard Lipschitz estimates. Condition~(iv) is satisfied by appending the Euclidean projection $\Pi_{B_{R_Y}}$ to the output layer of $\mathcal{V}_Y$. The projection is one-Lipschitz and places the complete valuative vector in the state-domain ball $B_{R_Y}$, as required
by Specification~\ref{bp:valuative}(iv).
\end{sloppypar}

As noted above, the valuative system modulates left attention $Q_L$, right awareness $W_R$, controller routing $\mathcal{R}_\controller$, memory consolidation, and action selection through the
five neuromodulatory signals. In addition, delayed credit assignment is implemented by eligibility
traces, cf. \citep{williams1992reinforce, mnih2015dqn}. For a parameter family $\theta_i$
belonging to action, attention, awareness, routing, or memory,
\[
z_i^{t+1} = \lambda_i z_i^t + \nabla_{\theta_i} \log\pi_i^t, \qquad \Delta\theta_i^t = \eta_i\,\delta^t\,z_i^t.
\]
Here $\nabla_{\theta_i}\log\pi_i^t$ denotes the REINFORCE score function evaluated at the
action $A^t$ actually taken: for the action policy this is
$\nabla_{\theta_{\mathrm{act}}} \log p(A^t\mid\tilde{Z}^t;\theta_{\mathrm{act}})$, and
analogously for each parameter family at the sample drawn at step $t$.
Both the trace update and the parameter update use $z_i^t$, the trace from step $t$, as opposed to the just-computed $z_i^{t+1}$.
The traces $z_i$ retain recent causes. The signal $\delta$ is delayed evaluation.
This allows the system to learn which attentional choices or awareness choices contributed to later success or error.

The five parameter families with their associated distributions $\pi_i$ and downstream
effects are:
\begin{itemize}
\item $\theta_{\mathrm{att}}$: attention weights in the symbolic module;
  $\pi_{\mathrm{att}}$ is the attention distribution over token positions; updates shift
  the precision-modulated attention distribution and hence $Q_L$.
\item $\theta_{\mathrm{aw}}$: awareness scoring weights in the geometric module;
  $\pi_{\mathrm{aw}}$ is the awareness distribution over edges; updates shift $W_R$.
\item $\theta_{\mathrm{rt}}$: routing weights in the controller; $\pi_{\mathrm{rt}}$ is
  the routing distribution over subsystems; updates shift $\mathcal{R}_\controller$.
\item $\theta_{\mathrm{act}}$: action policy weights; $\pi_{\mathrm{act}}=p(A\mid\cdot)$;
  updates shift the action distribution directly.
\item $\theta_{\mathrm{mem}}$: memory read and write weights;
  $\pi_{\mathrm{mem}}=p(\mathrm{read}\mid\cdot)$; updates shift which memories are
  consolidated or retrieved.
\end{itemize}

\begin{blueprint}[Policy distributions and credit signal]\label{bp:policy}
The policy distributions $\pi_i$ and credit signal $\delta$ belong to the admissible class
$\mathfrak{C}_\pi$ if:
\begin{enumerate}[label=(\roman*)]
\item \emph{Piecewise smoothness.} Each
$\pi_i(\cdot \mid Z;\theta_i)$ is strictly positive, Lipschitz in $(Z, \theta_i)$, and
continuously differentiable in $\theta_i$ on the compact admissible domain. Its score field
\[
s_i(Z,\theta_i) := \nabla_{\theta_i}\log\pi_i(\cdot\mid Z; \theta_i)
\]
is globally Lipschitz and bounded on that domain.
\item \emph{Trace decay.} $\lambda_i\in(0,1)$ for each $i$.
\item \emph{Bounded credit.} $\delta$ is a Lipschitz function of the valuative state
  $\valuative$ taking values in $[-D,D]$ for some fixed $D>0$. The signal may be signed
  to represent reward prediction error. The bound $D$ is an architecture constant.
\end{enumerate}
\end{blueprint}

\textit{Existence argument.} Condition~(i) is satisfied by a softmax policy
$\pi_i=\mathrm{softmax}(\ell_i(Z,\theta_i))$ whenever the logit map $\ell_i$ is smooth with
bounded first and second derivatives on the compact admissible domain. The resulting policy is
strictly positive and its score field is bounded and Lipschitz. In case an explicit probability
floor is desired, set
\[
\pi_i^\varepsilon = (1-\varepsilon) \pi_i + \varepsilon\,\mathrm{Unif}, \qquad \varepsilon\in(0,1),
\]
which preserves smoothness and strict positivity\footnote{We note that sparsemax-based policy gradients require a separate nonsmooth score-dynamics formulation at support transitions and lie outside the
classical RFDE policy class as used in the main theorems.}.
Condition~(ii) is a parameter specification: $\lambda_i \in (0,1)$ is set by hand or by meta-learning. Condition~(iii) is satisfied by $\delta = \mu_{\mathrm{DA}}$, which lies in $(0,1)$ by Specification~\ref{bp:valuative}(ii) and is Lipschitz in $\valuative$ by Specification~\ref{bp:valuative}(iii). More generally, any reward prediction error $\delta  =r - \hat{r}$ clipped to a bounded range is admissible.

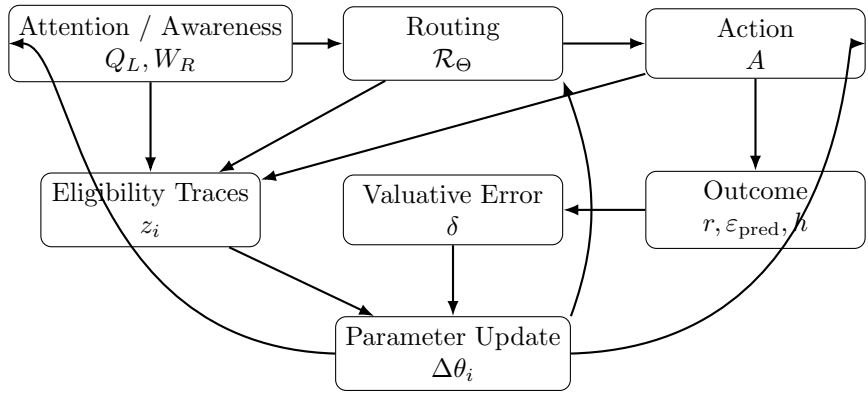
\begin{figure*}[t]
\centering
\resizebox{0.75\linewidth}{!}{%
\begin{tikzpicture}[
  box/.style={draw, rounded corners, align=center, minimum width=29mm, minimum height=9mm,
    font=\small},
  arrow/.style={-{Latex[length=2mm]}, thick},
  loop/.style={-{Latex[length=2mm]}, thick, bend left=18}
]
\node[box] (att) at (0,0) {Attention / Awareness\\$Q_L,W_R$};
\node[box] (route) at (4,0) {Routing\\$\mathcal R_\Theta$};
\node[box] (act) at (8,0) {Action\\$A$};
\node[box] (out) at (8,-2.2) {Outcome\\$r,\varepsilon_{\mathrm{pred}},h$};
\node[box] (val) at (4,-2.2) {Valuative Error\\$\delta$};
\node[box] (elig) at (0,-2.2) {Eligibility Traces\\$z_i$};
\node[box] (learn) at (4,-4.1) {Parameter Update\\$\Delta\theta_i$};
\draw[arrow] (att) -- (route);
\draw[arrow] (route) -- (act);
\draw[arrow] (act) -- (out);
\draw[arrow] (out) -- (val);
\draw[arrow] (att) -- (elig);
\draw[arrow] (route) -- (elig);
\draw[arrow] (act) -- (elig);
\draw[arrow] (elig) -- (learn);
\draw[arrow] (val) -- (learn);
\draw[arrow] (learn.west) .. controls (-4,-6) and (-1,0) .. (att.west);
\draw[arrow] (learn.north east) to[out=70, in=290] (route.south east);
\draw[arrow] (learn.east) .. controls (11,-6) and (9.2,0) .. (act.east);
\end{tikzpicture}}
\caption{Delayed credit assignment. Outcomes update, via eligibility traces, the prior choices of attention, awareness, routing, and action awareness, routing, and action.}
\end{figure*}

Each update $\Delta \theta_i = \eta_i \delta z_i$ modifies the weights used in the next
forward pass, closing the feedback loop from valuative error into the field variables
$Q_L$, $W_R$, $\mathcal{R}_\controller$, $A$, and $\memory$. The three precision and routing
fields are modulated by the neuromodulatory signals where $\mu_{\mathrm{ACh}}$ sharpens symbolic
attention as specified in Specification~\ref{bp:symbolic}(vi), $\mu_{\mathrm{NE}}$ shifts
geometric awareness logits, and $\mu_{\mathrm{DA}}$ shifts routing salience globally,
identifying $\delta = \mu_{\mathrm{DA}}$ as the common credit signal across parameter families,
weighted by individual learning rates $\eta_i$ and eligibility decays $\lambda_i$.

\textit{Viability for traces and policy parameters.}
The continuous-time eligibility trace equation is $\dot z_i=-\lambda_z z_i+s_i(Z_t)$, where
$s_i$ satisfies $\|s_i(Z_t)\|\le S_i$. At $\|z_i\|=R_{z,i}$,
the inner product satisfies
$\langle z_i,\dot z_i\rangle\le-\lambda_z R_{z,i}^2+R_{z,i}S_i\le0$
whenever $\lambda_z R_{z,i}\ge S_i$. 
The policy parameter equation is
\[
\dot\theta_i = \eta_i\bigl(\delta z_i-\lambda_{\mathrm{reg}}\theta_i\bigr),
\]
which is viable in $B_{R_{\theta,i}}$ whenever
\[
\lambda_{\mathrm{reg}}R_{\theta,i}\ge D R_{z,i}.
\]

\textit{Convergence characterization.} The RFDE analysis treats the score-function
estimator $\nabla_{\theta_i} \log\pi_i$ as a deterministic gradient field, corresponding
to the mean-field approximation in which eligibility traces track the expected gradient.
A stochastic treatment would model $\nabla_{\theta_i}\log\pi_i(A^t)$ as a random
variable and apply stochastic RFDE theory, cf. \citep{mao2007stochastic}. The deterministic
attractor result then applies to the expected dynamics. Under the specification conditions
and the mean-field approximation, the eligibility trace update $z_i^+ = \lambda_i z_i + s_i$ is contractive in $z_i$ with rate $\lambda_i \in (0,1)$. Combined with the radial-margin conditions above, the continuous $(Z, z, \theta)$ dynamics remain in the compact extended state domain, so
Theorem~\ref{thm:attractor} applies to the extended system. 

\textit{Coarse-graining.} The coarse-graining maps $\Phi_Y$ and $\Phi_\theta$ extract the
abstract classes $\mathcal{A}_Y$ (bounded Lipschitz valuative forcing $G_Y$ with relaxation rate
$\kappa_Y > 0$) and $\mathcal{A}_\theta$ (globally Lipschitz policy-learning flows on the compact
$(z,\theta)$ domain satisfying the trace and parameter radial-margin conditions above).

\section{Metacognition}

The metacognition module arises from self and world models that predict internal and external dynamics. These models provide second-order variables such as reliability, expected value of computation and internal conflict that are used to meta-learn the control of routing, learning, attention, and action. Meta-control is responsible for the choice of internal actions such as the request for additional perceptive (or other) feedback. Another feature of meta-control is the adjustment of learning rates and routing sparsity. The system thereby learns how to learn and how to regulate its own computation. In particular, this entails the controlled use of a self-world model to alter future cognition and behavior.

A world model predicts the effects of action on the environment, see e.g. \citep{ha2018world,
hafner2020dream}, and a self model predicts the effects of internal and external action on the
system's own future state. These models introduce various second-order variables, consisting of reliability, confidence, calibration, and internal conflict, among other second-order variables such as bias and the expected value of further computation. We recall below a possible simplified formalization of these concepts.

Let $\world$ predict external change and $\selfmodel$ predict internal change. Let $\rho_i$
denote the context-dependent reliability of subsystem $i$. The metacognitive layer is
summarized by
\[
\widehat{S}_{\mathrm{ext}}^+ = \world(S_{\mathrm{ext}},A,\state),
\quad
\widehat{\state}_{\mathrm{int}}^+ = \selfmodel(\state_{\mathrm{int}},A,B_\controller),
\]
\[
\rho_i^+ = (1-\alpha)\rho_i+\alpha\,f(\varepsilon_i),
\]
where $\varepsilon_i$ is the prediction error of subsystem $i$ and $f$ is a reliability
kernel specified below. The world model applies the geometric and symbolic
update rules under a hypothetical action $A$:
\[
\widehat{\rightstate}^+ = \mathcal{W}_{\mathrm{geom}}(\rightstate,A,B_\controller),
\]
\[
\widehat{\leftstate}^+ = \mathcal{W}_{\mathrm{sym}}(\leftstate,A,B_\controller).
\]
The world-model prediction error is $\varepsilon_{\mathcal{W}}^t = \widehat{\rightstate}^+
-\rightstate^{t+1}$, where $\widehat{\rightstate}^+$ was predicted at step $t$ using $Z^t$
and $A^t$, and $\rightstate^{t+1}$ is the observation available at the end of step $t$.
We obtain an error that is given as a stagewise-derived scalar at step $t$, available for computing
$\rho_i^{t+1}$ in Stage~8.2 in Section \ref{sec:formal_closure}. The normalized magnitude enters the neuromodulatory signals. The
self model predicts the next internal regulation state 
\[
\widehat{\state}_{\mathrm{int}}^+ =
\selfmodel\!\bigl((Q_L,W_R,\mathcal{R}_\controller,\valuative,\frontal,\memory),
\,A,\,B_\controller\bigr),
\]
with self-model error $\varepsilon_{\mathcal{S}}^t = \widehat{\state}_{\mathrm{int}}^+
- \state_{\mathrm{int}}^{t+1}$ with the same timing convention. Here, the index $i$ runs over the variables
\[
\{\mathcal{W}_{\mathrm{geom}},\mathcal{W}_{\mathrm{sym}},\mathcal{S},\Phi_{R\to L},\Phi_{L\to R}\}.
\]

\begin{blueprint}[Reliability update]\label{bp:reliability}
The reliability update belongs to the admissible class $\mathfrak{C}_\rho$ if:
\begin{enumerate}[label=(\roman*)]
\item \emph{Range.} $\rho_i \in [0,1]$ for all $i$ and all times.
\item \emph{Convex form.} $\rho_i^+ = (1 - \alpha) \rho_i + \alpha f(\varepsilon_i)$ for some
  fixed $\alpha\in(0,1)$ and some kernel $f \colon \mathbb{R}^{d_i} \to [0,1]$.
\item \emph{Radial monotonicity.} $f(\varepsilon) = \varphi(\|\varepsilon\|/\sqrt{d_i})$, where
$\varphi:[0,\infty)\to[0,1]$ is non-increasing and Lipschitz, and the prediction-error map
$Z\mapsto\varepsilon_i(Z)$ is Lipschitz on the admissible state domain. As a result,
$Z \mapsto f(\varepsilon_i(Z))$ is Lipschitz.
\end{enumerate}
\end{blueprint}

\textit{Existence argument.} Condition~(i) follows from condition~(ii).
Condition~(iii) restricts to radially non-increasing kernels. For instance, that includes
\[
f(\varepsilon) = \exp(-\|\varepsilon\|^2 d_i),  \ f(\varepsilon) = \left(1 + \frac{\|\varepsilon\|}{\sqrt{d_i}} \right)^{-1},
\]
and 
\[
f(\varepsilon) = \max\left(0, 1- \frac{\|\varepsilon\|}{\sqrt{d_i} \sigma}\right)
\] 
for a threshold $\sigma > 0$. The normalization by $\sqrt{d_i}$ here is essential, because without it the exponential kernel $\exp(-\|\varepsilon\|^2)$ collapses toward zero in high-dimensional spaces for any fixed per-dimension error, causing reliability to be insensitive to prediction quality.
With normalization, $f(\varepsilon)$ depends on the mean squared error per component, which
is the statistically appropriate measure of model accuracy and is bounded away from zero
for any accurate predictor regardless of $d_i$.

\textit{Viability.} The continuous-time reliability field is given by the relaxation flow
\[
  \dot\rho_i = \alpha_i\bigl(f(\varepsilon_i) - \rho_i\bigr), \qquad \alpha_i > 0.
\]
Since $f(\varepsilon_i) \in [0,1]$, at $\rho_i = 0$ one has $\dot\rho_i = \alpha_i f(\varepsilon_i) \ge 0$ and at $\rho_i = 1$ one has $\dot\rho_i = \alpha_i(f(\varepsilon_i) - 1) \le 0$.  Both box-cone
conditions, cf. Lemma~\ref{lem:viability}(ii), are satisfied, verifying Assumption~\ref{ass:A4} for
$\rho_i \in [0,1]$.

The reliability specified above has a feedback into routing and translation, i.e. the salience of subsystem $i$ is augmented by $\beta_\rho \rho_i$, and the translation gates are attenuated in proportion to translation reliability. As a result, unreliable channels are progressively closed and reliable subsystems receive stronger routing priority.

\section{Memory}

The system will contain various forms of memory, some of which are expected to be distributed properties such as the working memory\footnote{In fact, we can show that within the framework a kind of working memory arises out of the equations quite naturally, but we postpone the details to future investigations.}. By contrast, a more fixed, i.e. localized, type of memory is what is described here.
The consolidation of this type of memory occurs selectively in that the gate $g_M$ determines which symbolic and valuative content is written to the persistent field $M$, and the consolidation rate is calibrated by the executive state. The admissibility conditions stated below express this kind of selectivity.

The memory field $\memory$ accumulates content from the symbolic and valuative systems and
makes it available to all modules, and the update takes the form
\[
\memory^{t+1} = (1-g_\memory)\,\memory^t + g_\memory\,\Phi_\memory(\leftstate^{t+1},
\valuative^{t+1}),
\]
where $g_\memory\in(0,1)$ is a consolidation gate and $\Phi_\memory$ is a learned write
projection.

\begin{blueprint}[Memory update]\label{bp:memory}
The memory update belongs to the admissible class $\mathfrak{C}_\memory$ if:
\begin{enumerate}[label=(\roman*)]
\item \emph{Gated consolidation.}
  \[
  \memory^{t+1} = (1-g_\memory)\memory^t + g_\memory\Phi_\memory(\leftstate^{t+1}, \valuative^{t+1}),
  \]
  with $g_\memory\in[\varepsilon_M,1-\varepsilon_M]$ for some fixed
  $\varepsilon_M\in(0,\tfrac{1}{2})$.
\item \emph{Gate regularity.} $g_\memory$ is a Lipschitz function of $Z$.
\item \emph{Bounded write.} $\Phi_\memory$ is Lipschitz on the admissible input
domain and satisfies $\|\Phi_\memory(\cdot)\|\leq C_\memory$.
\end{enumerate}
\end{blueprint}

\textit{Existence argument.} The gated consolidation form above is a parameterized family that includes the GRU gated update and the write gate of differentiable neural computers as instances of this family. Any such form satisfies~(i) by construction. Condition~(ii) is satisfied whenever the gate is computed by a bounded differentiable function of $Z$, such as a sigmoid applied to a linear form. The source of
the gate signal, i.e. valuative state, content novelty, executive instruction, or any
combination, is an implementation choice within $\mathfrak{C}_\memory$. Condition~(iii) is satisfied by any linear write map with bounded operator norm on the stated bounded input domain; such a map is both bounded there and Lipschitz. This includes a matrix applied to the concatenation $[\leftstate^{t+1};\valuative^{t+1}]$ after mean pooling.

\textit{Viability.} The continuous-time memory field is given by the relaxation flow
\[
  \dot\memory = g_\memory(Z_t) \bigl(\Phi_\memory(Z_t)-\memory\bigr), \qquad g_\memory \ge 0.
\]
\begin{sloppypar}At $\|\memory\| = R_M$, the inner product $\langle\memory, \dot \memory\rangle =
g_\memory(\langle\memory, \Phi_\memory\rangle-\|\memory\|^2) \le
g_\memory(R_M^2-R_M^2) = 0$ by condition~(iii) and the Cauchy-Schwarz inequality. Hence
the ball $B_{R_M}$ is positively invariant, verifying Assumption~\ref{ass:A4} for the memory component.\end{sloppypar}

\section{Executive System}

The executive module implements a feedback loop that overrides stimulus-driven impulses in favor of planned, value-driven behaviors. The admissibility conditions consist of dissipativity and bounded range, which together ensure the executive state converges independently of the interfield coupling.

\begin{blueprint}[Executive state update]\label{bp:executive}
The executive update operator $\mathcal{P}$ belongs to the admissible class $\mathfrak{C}_P$
if:
\begin{enumerate}[label=(\roman*)]
\item \emph{Update form.} $\frontal^{t+1} = \mathcal{P}(\frontal^t, \{\Delta\theta_i^t\}, \valuative^{t+1})$ for some operator $\mathcal{P}$. The inputs are given by the executive state at $t$, the parameter deltas from Stage~8.3, and the valuative state from Stage~6 in Section \ref{sec:formal_closure}.
\item \emph{Regularity.} $\mathcal{P}$ is Lipschitz continuous in all arguments.
\item \emph{Bounded range.} $\|\mathcal{P}(\cdot)\| \leq R_P$ for all inputs in the state
  domain.
\item \emph{Dissipativity.} For every fixed input $\xi=(\{\Delta\theta_i\},\valuative)$ and all executive states $p,q$,
\[
\langle \mathcal P(p, \xi) - \mathcal P(q, \xi), \ p-q \rangle \le -\mu_P \|p - q\|^2.
\]
In the principal regime of Theorem~\ref{thm:stability}, the auxiliary input $\xi$ is frozen at its equilibrium value.
\end{enumerate}
\end{blueprint}

\textit{Existence argument.} A simple bounded discrete update is obtained by applying a
bounded activation to the concatenated inputs and then projecting the complete executive
vector onto $B_{R_P}$ with $\Pi_{B_{R_P}}$. For the continuous reaction field, a sufficient
construction is
\[
\mathcal P(p,\xi) = -\alpha_P p+\phi(W_Pp)+b(\xi),
\]
where $\phi$ is $1$-Lipschitz, $b$ is bounded and Lipschitz on the compact domain, and
\[
\alpha_P-\|W_P\|_{\mathrm{op}}\ge\mu_P>0.
\]
For fixed $\xi$, 
\begin{align*}
&\langle\mathcal P(p,\xi)-\mathcal P(q,\xi),p-q\rangle \le
-\bigl(\alpha_P-\|W_P\|_{\mathrm{op}}\bigr)\|p-q\|^2 \le
-\mu_P\|p-q\|^2.
\end{align*}
The required bound can be obtained via spectral normalization of $W_P$ and the appropriate choice of
$\alpha_P$.

\section{Module Interoperability}
\label{interoperability}

The output of each module enters a receiving module as a state-compatible input, after translation and gating. We formulate here the compatibility condition that is necessary for the interoperability of the modules. To this end, we utilize backpropagation through time\footnote{Given the limited biological plausibility of BPTT, a natural question is how to extend the eligibility-trace and modulatory learning mechanisms to the reentrant couplings.} as an instance of the discrete
adjoint method for the chosen history-buffer discretization.

Write the retarded functional differential equation in the form
\[
\dot Z(t) = F_{\vartheta}(Z_t,u(t)),
\]
\[
Z_t(\theta) = Z(t + \theta),\quad \theta \in [-\tau_{\max},0],
\]
on the admissible history space
\[
\mathcal{X} = C([-\tau_{\max},0],\Zspace),
\]
where \(\Zspace\) is the compact product state domain specified in the main body of the paper. The
parameter vector \(\vartheta\) collects all trainable parameters of all the relevant modules. For a fixed admissible value of \(\vartheta\), fix the assumptions on $F_{\theta}$ that are required for the well-posedness.

Replace the continuous history segment by a finite
history buffer. Fix a time step \(\Delta t>0\), choose \(K\in\mathbb N\) such that
\(K\Delta t\geq \tau_{\max}\), and set
\[
n_{R\to L}=\left\lfloor\frac{\tau_{R\to L}}{\Delta t}\right\rfloor,
\qquad
n_{L\to R}=\left\lfloor\frac{\tau_{L\to R}}{\Delta t}\right\rfloor .
\]
The augmented discrete state is
\[
\widetilde Z^t=(Z^t,Z^{t-1},\ldots,Z^{t-K})\in \Zspace^{K+1}.
\]
A single implementation step is a deterministic map
\begin{align*}
\widetilde Z^{t+1} &= \widetilde F_\vartheta(\widetilde Z^t,u^t,A^t) = (Z^{t+1},Z^t,\ldots,Z^{t-K+1}).
\end{align*}
The first coordinate \(Z^{t+1}\) is obtained by an explicit Euler step, an
operator-splitting step, or another consistent time-stepping rule for the RFDE.
In the explicit Euler case,
\[
Z^{t+1} =
Z^t+\Delta t\,F_\vartheta
\bigl(Z^t,Z^{t-n_{R\to L}},Z^{t-n_{L\to R}},u^t\bigr).
\]
In the stagewise implementation used in the formal closure section,
\(Z^{t+1}\) is computed by the ordered stage maps for the various updates. In
either case, the implementation is a finite-dimensional recurrent dynamical
system on \(Z^{K+1}\).

For the principal symbolic and geometric fields, interoperability requires the
updates to take the form
\[
H_L^{t+1} = \mathcal{T}_L \left(H_L^t,B_\Theta^t,C_{R\to L}^t,Q_L^{t+1},P^t,M^t\right),
\]
\[
X_R^{t+1} = \mathcal{G}_R \left(X_R^t,B_\Theta^t,C_{L\to R}^t,W_R^{t+1},Y^t,M^t \right).
\]
The delayed interconnector packets are
\[
C_{R\to L,\ell}^t = g_{R\to L}^t \sum_{i\in V} \alpha_{\ell i}(Z^t) K_\vartheta(\ell,i) e_i^{\,t-n_{R\to L}},
\]
and
\[
C_{L\to R,i}^t = g_{L\to R}^t \sum_{\ell=1}^T \beta_{i\ell}(Z^t) K_\vartheta(\ell,i)^\top h_\ell^{\,t-n_{L\to R}}.
\]
Here
\[
h_\ell\in\mathbb R^{d_L},\quad e_i\in\mathbb R^{d_R},\quad K_\vartheta(\ell,i)\in\mathbb R^{d_L\times d_R}.
\]
Therefore
\[
C_{R\to L}^t\in\mathbb R^{T\times d_L}, \qquad
C_{L\to R}^t\in\mathbb R^{|V|\times d_R}.
\]
This compatibility constitutes the first part of interoperability. The second
part is compatibility with the normalizations that are used in the module specifications.
Recall that the symbolic precision field
\[
Q_L\in[\varepsilon_Q,R_Q]^{T\times T}.
\]
The resulting floored weighted graph Laplacian contributes to the
dissipativity constant \(\mu_L\) through the spectral gap of the symmetrized
symbolic conductance graph. The geometric awareness field \(W_R\) is normalized
differently, in that for each target node it lies in a simplex of incoming edge weights.
A positive geometric spectral gap is obtained either by connected active support or by a fixed sparse residual backbone. The routing matrix \(R_\Theta\) is row-stochastic, so interconnector
gates such as \(g_{R\to L}\) and \(g_{L\to R}\), being entries of \(R_\Theta\),
lie in \([0,1]\). The neuromodulatory vector satisfies \(\mu\in(0,1)^5\), the reliability variables satisfy \(\rho_i\in[0,1]\), and ball-valued variables remain in their prescribed compact balls. Thus each module exports bounded normalized signals and receives bounded normalized signals in the receiving module's admissible state space.

The third part of interoperability is control of the reentrant coupling budget.
Recall that the kernel satisfies the Hilbert-Schmidt bound
\[
\|K_\vartheta\|_{\mathrm{HS}}^2 = \sum_{\ell=1}^T\sum_{i\in V} \|K_\vartheta(\ell,i)\|_F^2
\leq C_K^2.
\]
Here the simplex-valued weights \(\alpha_{\ell i}\) and \(\beta_{i\ell}\) redistribute
this coupling across token-node pairs without enlarging the coupling budget.
The controller gates attenuate the translated packets. As a consequence, all the types of kernel discussed in appendix \ref{app:interconnector} may be used, provided that the induced operator remains Lipschitz and satisfies the same Hilbert-Schmidt bound. In the closed principal regime, the constant \(C_K\)
is the coupling quantity entering the small-gain condition. We next recall the general adjoint principle used below. Note that an adjoint method is a gradient-computation method for an optimization problem constrained by a state equation. Let
\[
R(y,\vartheta)=0
\]
be a residual equation between state spaces, and let \(J(y,\vartheta)\) be an
objective. The reduced objective is
\[
\widehat J(\vartheta) = J(y(\vartheta), \vartheta),
\]
where \(y(\vartheta)\) is determined by the residual equation. Differentiating the constraint yields
\[
D_y R(y,\vartheta)\,\delta y + D_\vartheta R(y,\vartheta)\, \delta \vartheta = 0.
\]
The adjoint method now introduces a dual variable \(p\) and chooses it to solve
\[
D_yR(y, \vartheta)^\ast p = D_y J(y, \vartheta).
\]
Then the first variation of the reduced objective is
\[
D\widehat J(\vartheta)\, \delta \vartheta = D_\vartheta J(y,\vartheta)\,\delta\vartheta
-
\left\langle p, D_\vartheta R(y,\vartheta)\,\delta\vartheta \right\rangle.
\]
Thus the dependence of the objective on the state variation is transferred to a
dual equation involving the adjoint of the linearized state operator. This is the
standard adjoint-state principle in constrained optimization which in discrete time turns into the discrete adjoint method. Standard references include
\cite{HinzePinnauUlbrichUlbrich2009} for the adjoint-state method in constrained
optimization and \cite{GriewankWalther2008} for the reverse-mode automatic
differentiation viewpoint.

We now specialize this principle to the RVF history-buffer implementation. The
state equation is the recurrence
\[
\widetilde Z^{t+1} = \widetilde F_\vartheta(\widetilde Z^t,u^t,A^t), \qquad t=0,\ldots,N-1.
\]
Equivalently, define residuals
\begin{align*}
R_t(\widetilde Z^{t+1}, \widetilde Z^t,\vartheta) &= \widetilde Z^{t+1} - \widetilde F_\vartheta(\widetilde Z^t,u^t,A^t) = 0.
\end{align*}
The full discrete state is given by
\[
\mathcal Z = (\widetilde Z^0,\widetilde Z^1,\ldots,\widetilde Z^N),
\]
and the collection of equations \(R_t=0\) is the residual system
\[
R(\mathcal Z,\vartheta)=0.
\]

Let the finite-horizon objective be given by
\[
\mathcal J(\mathcal Z,\vartheta) = \Phi(\widetilde Z^N) + \sum_{t=0}^{N-1} \ell_t(\widetilde Z^t,\vartheta).
\]
Applying the adjoint principle to this finite-dimensional residual system yields
a backward recursion for covectors
\[
\lambda^t \in (T_{\widetilde Z^t}Z^{K+1})^\ast.
\]
Setting
\[
D_t = D_{\widetilde Z}\widetilde F_\vartheta(\widetilde Z^t,u^t,A^t),
\]
\[
B_t = D_\vartheta\widetilde F_\vartheta(\widetilde Z^t,u^t,A^t),
\]
the terminal adjoint is
\[
\lambda^N = D_{\widetilde Z} \Phi(\widetilde Z^N),
\]
with the evident additional term \(D_{\widetilde Z}\ell_N\) if a terminal stage
loss is included. For \(t=N-1, \ldots, 0\), the discrete adjoint equation is
\[
\lambda^t = D_{\widetilde Z}\ell_t(\widetilde Z^t,\vartheta) + D_t^\ast\lambda^{t+1}.
\]
The gradient of the reduced objective is
\begin{equation*}
\begin{split}
\frac{d\mathcal J}{d\vartheta}
&= \sum_{t=0}^{N-1}\!\left(D_\vartheta \ell_t(\widetilde Z^t,\vartheta)+B_t^\ast\lambda^{t+1}\right)
+D_\vartheta\Phi(\widetilde Z^N).
\end{split}
\end{equation*}
Thereby we have the discrete adjoint method for the chosen map
\(\widetilde F_\vartheta\). 
Next, suppress exogenous inputs and write the first coordinate as
\[
x^{t+1} = f_\vartheta(x^t, x^{t-n_1}, x^{t-n_2}),
\]
with
\[
\widetilde x^t = (x^t, x^{t-1}, \ldots, x^{t-K}).
\]
Then
\[
D_t = \begin{pmatrix}
A_0^t & A_1^t & \cdots & A_K^t\\
I     & 0     & \cdots & 0\\
0     & I     & \cdots & 0\\
\vdots&       & \ddots & \vdots\\
0     & 0     & \cdots & 0
\end{pmatrix},
\qquad
A_j^t = \frac{\partial f_\vartheta}{\partial x^{t-j}}.
\]
In the simplest fixed-delay case, we have that
\[
A_j^t=0 \quad \text{unless} \quad j \in \{0,n_1,n_2\},
\]
apart from additional nonzero blocks introduced by interpolation or by stagewise
dependencies. Below the first row, the identity blocks shift the history. The adjoint recursion therefore propagates covectors backward through both
the current-state dependence and the delayed read branches. In
component notation, a state \(x^s\) receives adjoint contributions not only from
the next time \(s+1\), but also from later times at which \(x^s\) is read as a
delayed state:
\begin{align*}
\bar x^s &= D_{x^s}\ell_s+(A_0^s)^\ast\bar x^{s+1} +(A_{n_1}^{s+n_1})^\ast \bar x^{s+n_1+1}
+ (A_{n_2}^{s+n_2})^\ast\bar x^{s+n_2+1} + \cdots .
\end{align*}

For the principal symbolic-geometric loop this has a more concrete form. To wit, in an Euler
discretization of the fixed-coupling principal regime, the delayed coupling contributes
\[
H_L^{t+1} = \cdots + \Delta t\, K_\vartheta X_R^{t-n_{R\to L}},
\]
\[
X_R^{t+1} = \cdots + \Delta t\, K_\vartheta^\ast H_L^{t-n_{L\to R}}.
\]
Thus
\[
\frac{\partial H_L^{t+1}}{\partial X_R^{t-n_{R\to L}}} = \Delta t\,K_\vartheta,
\quad \frac{\partial X_R^{t+1}}{\partial H_L^{t-n_{L\to R}}} = \Delta t\, K_\vartheta^\ast.
\]
The same interfield channels carry delayed reentrant signals forward and adjoint variables backward. In
the adaptive regime, additional Jacobian terms arise from the state dependence of fields and valuative modulation. The Lipschitz assumptions control these terms. They must be included in the numerical adjoint whenever they are trainable.

We also obtain the coupling gradient directly from the adjoint recursion. If
\[
H_L^{t+1} = \cdots + \Delta t\,K_\vartheta X_R^{t-n_{R\to L}},
\]
then the contribution to the gradient of \(K_\vartheta\), in the Hilbert-Schmidt pairing, is
\[
\Delta t \sum_t \lambda_H^{t+1}\otimes X_R^{t-n_{R\to L}}.
\]
Similarly, from
\[
X_R^{t+1} = \cdots + \Delta t\,K_\vartheta^\ast H_L^{t-n_{L\to R}},
\]
one obtains
\[
\Delta t \sum_t H_L^{t-n_{L\to R}}\otimes \lambda_X^{t+1}.
\]
Thus the reentrant kernel is updated by correlations between adjoint
errors at receiving times and field states at delayed sending times. 

We note that the adjoint calculus depends on the numerical realization. Given different step sizes, delay interpolation rules, explicit or implicit treatment of diffusion, and operator-splitting orders entail different maps \(\widetilde F_{\vartheta}\), and hence different discrete adjoint equations. In that sense, BPTT is the adjoint of the chosen time-stepping scheme. For implementation,
the relevant fact is that BPTT computes the exact gradient of the finite rollout
actually used for training, up to the usual numerical and automatic-differentiation
errors.

By forward admissibility is meant that, for fixed \(\vartheta\), the map
\(\widetilde F_\vartheta\) sends \(Z^{K+1}\) into \(Z^{K+1}\), is Lipschitz, and
preserves the normalized module interfaces. By training admissibility we mean
that the parameter update produced from the BPTT gradient remains inside the
parameter region for which forward admissibility and the stability estimates
hold. If gradient descent is unconstrained it may reduce the task loss and increase \(\|K_\vartheta\|_{\mathrm{HS}}\), as well as reduce \(\mu_L\) or \(\mu_R\). In addition, it may destroy a
conductance floor, violate a simplex constraint, or increase a Lipschitz
constant beyond the admissible range.

A soft-constrained training objective may therefore be written as
\begin{equation*}
\begin{split}
\mathcal L_{\mathrm{train}} &= \mathcal L_{\mathrm{task}}
+\lambda_K\bigl(C_K(\vartheta)^2-\mu_L(\vartheta)\mu_R(\vartheta)\bigr)_+
+ \lambda_Z \mathcal V_Z(\vartheta) + \lambda_{\mathrm{Lip}} \mathcal V_{\mathrm{Lip}}(\vartheta),
\end{split}
\end{equation*}
where \((r)_+=\max\{r,0\}\). The first term penalizes violation of the Schur
small-gain margin. The term \(\mathcal V_Z\) collects violations of compact-domain
constraints, including ball, box, and simplex constraints. The term
\(\mathcal V_{\mathrm{Lip}}\) collects spectral-norm or other Lipschitz controls.
As an alternative, we may consider projected BPTT where after each optimizer step,
perform the following projection. First \(K_\vartheta\) is projected onto the Hilbert-Schmidt ball, then restore the floored conductance parameterization of \(Q_L\) and project awareness as well as routing variables onto their simplices. Then project ball-valued parameters or outputs onto their
prescribed compact domains.

The interoperability principle can now be summarized. A trainable module family
is interoperable if, for every admissible history and input, the induced discrete
map satisfies:
\begin{enumerate}[label=(I\arabic*)]
\item translated packets have the receiving module's dimension;
\item $\widetilde F_\vartheta(\Zspace^{K+1})\subseteq \Zspace^{K+1}$;
\item $\widetilde F_\vartheta$ is Lipschitz on $\Zspace^{K+1}$;
\item $Q_L,W_R,R_\Theta,\mu,\rho_i$ satisfy their specified normalizations;
\item $\|K_\vartheta\|_{\mathrm{HS}}\leq C_K$, and in the closed principal regime $C_K^2<\mu_L\mu_R$;
\item BPTT updates are penalized or projected so that (I1)--(I5) remain valid.
\end{enumerate}

Under these conditions, we obtain a time-discrete realization.


\section{Coarse Graining}\label{sec:coarsegrained}

A coarse-graining replaces each of the specification conditions with an abstract operator class defined by dissipativity constants, symmetry, compact viability, and spectral conditions. 

\subsection*{Coarse-graining maps}

For each specification $i$, the \emph{coarse-graining map}
$\Phi_i\colon\mathcal{C}_i\to\mathcal{A}_i$
extracts the field-theoretic content of the admissibility class $\mathcal{C}_i$ as an
abstract operator class $\mathcal{A}_i$.  Each $\mathcal{A}_i$ specifies: (a) the function
space the operator acts on; (b) its dissipativity or spectral constant; (c) its symmetry
group. 

\subsection*{Abstract operator classes}

\begin{itemize}
\item $\mathcal{A}_L$: dissipative operators on $L^2(\GL,\mathbb{R}^{d_L})$,
  with dissipativity constant $\mu_L>0$.
\item $\mathcal{A}_R$: dissipative operators on $L^2(\GR,\mathbb{R}^{d_R})$ with
  $\mathrm{SE}(d)$-invariance in the scalar feature sector, $\mathrm{Aut}(\GR)$-equivariance, and dissipativity constant $\mu_R>0$.
\item $\mathcal{A}_\Kcal$: Hilbert-Schmidt bipartite operators
  \begin{equation*}
    \Kcal\colon L^2(\GR,\mathbb{R}^{d_R})\to L^2(\GL,\mathbb{R}^{d_L}),\quad
    \|\Kcal\|_{\mathrm{HS}}\leq C_\Kcal.
  \end{equation*}
\item $\mathcal{A}_{Q_L}$: positive precision fields $Q_L$ whose symmetric conductance
$\overline Q_L$ has a uniform positive floor on a fixed connected support.
\item $\mathcal{A}_{W_R}$: simplex-valued directed awareness fields $W_R$ whose
symmetrization $\overline W_R$ has connected support and defines the geometric diffusion
operator.
\item $\mathcal{A}_{R_\Theta}$: row-stochastic weighted adjacency matrices on the subsystem
graph $G_C$ with $\|R_\Theta\|_\infty=1$ and Lipschitz state dependence.
\item $\mathcal{A}_Y$: bounded Lipschitz valuative forcing functions $G_Y$ with
relaxation rate $\kappa_Y>0$.
\item $\mathcal{A}_\theta$: globally Lipschitz policy-learning flows on the compact
$(z,\theta)$ domain satisfying the trace and parameter radial-margin conditions.
\item $\mathcal{A}_P$: dissipative operators on $\mathbb{R}^{n_P}$ with dissipativity
  constant $\mu_P>0$.
\item $\mathcal{A}_\rho$: Lipschitz reliability relaxation fields on $[0,1]$ of the form
$\dot\rho_i=\alpha_i(f(\varepsilon_i)-\rho_i)$ with Lipschitz
$Z\mapsto f(\varepsilon_i(Z))$.
\item $\mathcal{A}_M$: Lipschitz gated memory-relaxation fields on $B_{R_M}$ of the form
$\dot M=g_M(Z)(\Phi_M(Z)-M)$ with bounded Lipschitz $g_M$ and $\Phi_M$.
\end{itemize}
Any tuple $(T_L,G_R,\Kcal,Q_L,W_R,R_\Theta,G_Y,\pi_\theta,P)\in\prod_i\mathcal{A}_i$ satisfying
the RFDE assumptions defines a well-posed RFDE with a compact attractor. Selecting operators from all
the classes defines a block vector field
\[
\dot Z(t)=\mathcal F(Z_t,u^*)
\]
on the compact phase domain \(\mathcal Z\) with the required properties.

\bibliographystyle{plainnat}

\end{document}